\documentclass{birkjour_t2}
\usepackage{amssymb}
\usepackage{amsmath}
\usepackage{bm} 
\usepackage{graphicx}
\usepackage{graphics}
\usepackage{caption2}
\usepackage{color}
\usepackage{psfrag}
\usepackage{amsmath}
\usepackage{mathtools}
\usepackage{algorithm}
\usepackage{algorithmic}
\usepackage{subfigure}

\usepackage[numbers,sort&compress]{natbib}

\usepackage{amscd}
\usepackage{amsfonts}
\usepackage{float}
\usepackage{latexsym}

\usepackage{lscape}
\usepackage{mathrsfs}
\usepackage{comment}
 \newtheorem{theorem}{Theorem}[section]
 \newtheorem{lemma}[theorem]{Lemma}
\newtheorem{prop}[theorem]{Proposition}
 \theoremstyle{definition}
 \newtheorem{defn}[theorem]{Definition}
 \theoremstyle{remark}
\newcommand{\RNum}[1]{\uppercase\expandafter{\romannumeral #1\relax}}

\begin{document}

%-------------------------------------------------------------------------
% editorial commands: to be inserted by the editorial office
%
%\firstpage{1} \volume{228} \Copyrightyear{2004} \DOI{003-0001}
%
%
%\seriesextra{Just an add-on}
%\seriesextraline{This is the Concrete Title of this Book\br H.E. R and S.T.C. W, Eds.}
%
% for journals:
%
%\firstpage{1}
%\issuenumber{1}
%\Volumeandyear{1 (2004)}
%\Copyrightyear{2004}
%\DOI{003-xxxx-y}
%\Signet
%\commby{inhouse}
%\submitted{March 14, 2003}
%\received{March 16, 2000}
%\revised{June 1, 2000}
%\accepted{July 22, 2000}
%
%
%
%---------------------------------------------------------------------------
%Insert here the title, affiliations and abstract:
%

\title[f-WANs fPDEs]{Fractional weak adversarial networks for the stationary fractional advection dispersion equations }

%----------Author 1
\author[Feng]{Dian Feng}

\address{School of Mathematical Sciences\\
Fudan University\\ 
Shanghai 200433\\
China}

\email{dfeng19@fudan.edu.cn}

%\thanks{This work was completed with the support of our \TeX-pert.}
%----------Author 2
\author[Yang]{Zhiwei Yang$^*$}

\address{School of Mathematical Sciences\\
Fudan University\\ 
Shanghai 200433\\
China}

\email{zhiweiyang@fudan.edu.cn}

\author[Zou]{Sen Zou}

\address{School of Mathematical Sciences\\
Fudan University\\ 
Shanghai 200433\\
China}

\email{szou18@fudan.edu.cn}

\let\thefootnote\relax\footnotetext{$^*$Corresponding author: Zhiwei Yang}
\thanks{This research was    funded     by the China Postdoctoral Science Foundation    2022M720809.
}
%----------classification, keywords, date


\date{today}
%----------additions
%\dedicatory{To my boss}
%%% ----------------------------------------------------------------------

\begin{abstract}

In this article, we propose the fractional weak adversarial networks (f-WANs) for the stationary fractional advection dispersion equations (FADE) based on their weak formulas. This enables us to handle less regular solutions for the fractional equations. To handle the non-local property of the fractional derivatives, convolutional layers and special loss functions are introduced in this neural network. Numerical experiments for both smooth and less regular solutions show the validity of f-WANs. 
\end{abstract}

%%% ----------------------------------------------------------------------
\maketitle
%%% ----------------------------------------------------------------------
%\tableofcontents
\section{Introduction}

Recently, there has been a significant surge of interest in fractional differential equations, due to their widespread application in modeling complex phenomena that exhibit memory and long-range dependence across various fields, such as turbulent flow \cite{Chen, ShlWes} and viscoelastic constitutive law \cite{Mainar}. This increasing interest has led to a lot of research focused on creating and analyzing methods to solve these equations \cite{KilSri, Pod, Weilbe}. 
In this paper, we propose a neural network for the numerical solution of the steady-state fractional advection dispersion equation (FADE) in high-dimensional spaces. The FADE has its application in modeling physical phenomena of anomalous diffusion \cite{JinRun, MetKla}. The theoretical well-posedness and numerical methods of this equation can be found in \cite{ErvRoo, WanYan, ErvRoo1, Roop, ZhuLiu}.
However, the non-local nature of fractional differential operators implies that the  coefficient matrix associated with FADE tends to be dense or even full during discretization. 
Moreover, as the dimensionality increases, traditional methods inevitably require discretizing the solution domain, resulting in a substantial increase in the storage cost for computers. 
To overcome the numerical difficulties of nonlocal property and dimensionality, we explore the neural network solutions for space-fractional differential equations.

Over the last few years, various neural network structures have been developed to solve partial differential equations (PDEs) \cite{EYu, FanBoh, FanLin, KhooLu, KhoYin, LiLu, LongLu, NabMei, RaiPer, ShaTan, SirSpi, ZanBao}. Analogous to the difference between supervised and unsupervised learning algorithms, the neural networks for solving PDEs  can be roughly  divided into two categories.  The first category of neural networks approximate the solution map governed by specified PDEs directly by training over the large set of boundary/initial conditions and the corresponding solutions \cite{FanBoh, FanLin, KhooLu, KhoYin, LiLu, LongLu, ShaTan}. An example of this neural network structure is the BCR-Net proposed in \cite{FanBoh} based on wavelet transform. These neural networks design specific structure due to certain properties of the porposed solution map, but the computational costs when generating the training data is significantly large. The second category of neural networks aims to approximate equations or the deformation of equations under specific boundary/initial conditions \cite{EYu, RaiPer, ZanBao}. Physics-informed neural networks (PINNs) serve as a prime example of this category \cite{RaiPer}. By representing the solution as a neural network and then incorporating the equation and boundary conditions to derive the loss function for training, this method shows its efficienticy in deployment for various types of equations.

When it comes to fractional differential equations, a so-called fractional PINNs (fPINNs) were introduced in \cite{PangLu} to extend PINNs into fractional cases. In addition to the article \cite{PangLu}, \cite{GuoWu} proposed the Monte Carlo fPINNs method to address the computational challenges posed by the high dimensionality of fPINNs. Due to the $L_2$ formalution of the loss function, these PINNs based methods have good performence on classical solutions. However, the PINNs based methods may not perform well for less smooth solutions \cite{ZanBao}.

To fill this research gap where solution is less regular, we propose a novel neural network structure, called fractional weak adversarial networks (f-WANs), solving the weak formula of the fractional differential equations using the generative adversarial structure. Such a structure has been applied to elliptic partial differential equations and related inverse problem in \cite{ZanBao}.To be more precise, we use the Monte Carlo sampling to obtain the parameterization of the weak solution and the test function in the weak formulation of the fraction equations as two neural networks, which are trained alternately in a confrontational manner to obtain the solution of the given minimax problem. There are several advantages to our approach. Firstly, it can handle situations where the classical solution does not exist. Secondly, our method can overcome the ``curse of dimensionality" as it uses the Monte Carlo sampling method within the region, thereby avoiding the grid division brought by traditional numerical methods.

The structure of the paper is as follows: In Section \ref{Sect:pre}, we present the formulation of the problems addressed in this paper and establish their uniqueness. Section \ref{Sect:WANf} provides an introduction to the WAN framework and outlines the proposed algorithm for training neural networks. In Section \ref{Sect:NumerExper}, we present the neural network architectures for the solution function and the test function, along with numerical examples to illustrate their effectiveness. Finally, Section \ref{Sect:Conc} concludes the paper.

\section{Preliminaries and problem formulation}\label{Sect:pre}
In this section, we introduce the definitions of fractional operators and fractional order spaces. These concepts will serve as the foundation for deriving the weak form of the fractional order equation using the variational approach. The conditional uniqueness of our problem is proved.
\subsection{Fractional order derivative}
For completeness we first introduce the definition of Riemann-Liouville fractional integral:

\begin{defn}\label{Def:R-L_int} \rm 
  (Riemann-Liouville Fractional Integral \cite{Pod}). Let $u$ be a $L^1$ function defined on $(a, b)$, and $\alpha>0$. Then the left and right Riemann-Liouville fractional integrals of order $\alpha$ are defined as
\begin{equation*}
\begin{aligned}
J^{\alpha} u(x) & :=\frac{1}{\Gamma(\alpha)} \int_a^x(x-w)^{\alpha-1} u(w) \mathrm{d} w,\\
 J^{\alpha}_{-} u(x) & :=\frac{1}{\Gamma(\alpha)} \int_x^b(w-x)^{\alpha-1} u(w) \mathrm{d} w,
\end{aligned}
\end{equation*}
where $\Gamma(\cdot)$ is the Gamma function.
\end{defn}

The Riemann-Liouville fractional derivatives are given by taking normal derivatives of the Riemann-Liouville fractional integrals:
\begin{defn} \rm   (Riemann-Liouville Fractional Derivative \cite{Pod}). For $\alpha\in(0,1)$, the left and right Riemann-Liouville fractional derivatives of order $\alpha$ are defined as follows:
\begin{equation}
\begin{split} 
&\partial^{\alpha} u(x):=\frac{\partial}{\partial x}J^{1-\alpha} u(x)=\frac{1}{\Gamma(\alpha)} \frac{\partial}{\partial x} \int_{a}^{x} u(s)(x-s)^{-\alpha} \mathrm{d} s, \\
& \partial_{-}^{\alpha} u(x):= - \frac{\partial}{\partial x} J^{1-\alpha}_{-} u(x)=\frac{-1}{\Gamma(\alpha)} \frac{\partial}{\partial x} \int_{x}^{b} u(s)(s-x)^{-\alpha} \mathrm{d} s.
\end{split}
\end{equation}
\end{defn}

\subsection{Stationary fractional advection dispersion equations}
In this paper we assume the domain of definition is a rectangular domain $\Omega= \left(\underline{x}_{i}, \overline{x}_{i} \right)^n\subset\mathbb{R}^n$. We use the notation $\underline{{\ \cdot\ }}$ and  $\overline{{\ \cdot\ }}$ to denote the lower and upper boundaries in each direction. Now we are ready to give the Dirichlet problem for the stationary FADE:
\begin{equation}\label{frac_prob}\left\{
\begin{aligned}
	\mathcal{L} u & =f, & & \text { in } \Omega, \\
	u & =g, & & \text { on } \partial \Omega,
\end{aligned}\right.
\end{equation}
where
\begin{equation}\label{frac_oper}
    \mathcal{L} u:= -\sum_{i=1}^{n}\partial_{x_{i}}\left(p_{i}   J^{\alpha}_{\underline{x}_i}+q_{i} J^{\alpha}_{\overline{x}_i-} \right)\partial_{x_{i}} u,
\end{equation}
with constants $\alpha\in (0,1)$, and $\sum\limits_{i=1}^{n}(p_{i}+q_{i})=1$ with $ p_{i},q_{i}>0$ for $1\leq i \leq n$. Here $J^{\alpha}_{\underline{x}_i}$ denotes the left fractional integral over the interval $(\underline{x}_i,x_i)$, while $J^{\alpha}_{\overline{x}_i-}$ represents the right fractional integral over the interval $(x_i, \overline{x}_i)$. 

Now we introduce some function spaces related to the fractional derivatives. Recall the usual fractional order Sobolev spaces $H^{\alpha}(0,T)$, see e.g. \cite{AdaFou}. Following the notation in \cite{KubRys}, for any constant $\alpha\in(0,1)$, we define the Banach spaces
$$
H_{\alpha}(0, T):=\left\{
\begin{aligned}
&\left\{v\in H^{\alpha}; v(0)=0\right\}, && \frac{1}{2}<\alpha<1, \\
&\left\{v \in H^{\frac{1}{2}}(0, T);\, \int_{0}^{T} \frac{|v(t)|^{2}}{t} d t<\infty\right\}, && \alpha=\frac{1}{2}, \\ &H^{\alpha}(0, T), &&0<\alpha<\frac{1}{2}
 \end{aligned}\right.
$$
with the norm
$$
\|v\|_{H_{\alpha}(0, T)}=\left\{
\begin{aligned}
&\|v\|_{H^{\alpha}(0, T)}, && \alpha \neq \frac{1}{2}, \\
&\left(\|v\|_{H^{\frac{1}{2}}}^{2}+\int_{0}^{T} \frac{|v(t)|^{2}}{t} d t\right)^{\frac{1}{2}}, &&\alpha=\frac{1}{2}.
\end{aligned}\right.
$$

Next we define the Sobolev spaces given by Riemann-Liouville fractional derivative for $s>0$:
$$\hat{H}_{s}(\Omega)=\left\{ v\in L^2(\Omega), \partial^{s}_{x_{i}} v\in L^2(\Omega), \,\text{for all }\, i=1,\dots,n\right\},$$
with the norm
$$
\|v\|^2_{\hat{H}_{s}(\Omega)}:=\sum_{i=1}^n\|\partial^{s}_{x_{i}} v\|^2_{L^2(\Omega)}.
$$
The next proposition states that the $L^2$-norm of a function can be bounded by its $H_{1-\alpha}$-norm.

\begin{prop}\label{prop:norm_ine}
Given any $\alpha\in(0,1)$, for any $u\in H_{1-{\alpha}}(\underline{x},\overline{x})$, the following inequality holds
\begin{equation}
    \|u\|_{L^2(\underline{x},\overline{x})} \leq C  \|J^{\alpha} \partial_x u \|_{L^2(\underline{x},\overline{x})}.
\end{equation}
where $C>0$ is a constant.
\end{prop}

\begin{proof}
Denote $v:=J^{\alpha} \partial_x u$, by applying the fractional integral both sides we obtain
 $u=J^{1-\alpha}  v$. 
Therefore it's sufficient to prove  $\|J^{1-\alpha} v\|_{L^2(\underline{x},\overline{x})} \leq C \|  v \|_{L^2(\underline{x},\overline{x})} $. 
Using Young's inequality \cite{AdaFou},
$$
\begin{aligned}
\|J^{1-\alpha} v\|_{L^2(\underline{x},\overline{x})}&=\| \frac{1}{\Gamma(1-\alpha)} \int_{x_l}^{x} (x-w)^{-\alpha}v(w)\mathrm{d}w \|_{L^2(\underline{x},\overline{x})} \\
&= C \|x^{-\alpha}*v\|_{L^2(\underline{x},\overline{x})}\\
&\leq C \|x^{-\alpha}\|_{L^1(\underline{x},\overline{x})} \|v\|_{L^2(\underline{x},\overline{x})} \\
& \leq C \|v\|_{L^2(\underline{x},\overline{x})},
\end{aligned}
$$
which concludes the proof.
\end{proof}
Now we state a conditional uniqueness result of \eqref{frac_prob}.
\begin{theorem}\label{them:uniqu_the}
If the equation \eqref{frac_prob} admits a non-trivial solution $u\in \hat{H}_{1}(\Omega)$ , then $u$ is unique in the sense of $\hat{H}_{1-\alpha}(\Omega)$ for any fixed $\alpha\in (0, 1)$, i.e. for any $u_1, u_2 \in \hat{H}_{1}(\Omega)$ satisfy \eqref{frac_prob}, $$\|u_1-u_2\|_{\hat{H}_{1-\alpha}(\Omega)}=0.$$
\end{theorem}

\begin{proof}
First, we consider the homogeneous boundary condition $g=0$. To prove the uniqueness result, it's enough to prove
\begin{equation}\label{uniq_ine}
\|u\|_{\hat{H}_{1-\alpha}(\Omega)} \leq C\|f\|_{L^2(\Omega)}. 
\end{equation}
We denote coordinate notation as $\boldsymbol{x}:=(\hat{x}_i,x_i)$, here $\hat{x}_i = (x_1, \cdots, x_{i-1},x_{i+1},\cdots, x_n)$. And the domain with respect to $\hat{x}_i$ is denoted by $\Omega^{\prime}$.
Multiplying both sides of \eqref{frac_prob} with $u$, the integration by parts gives that

$$
\begin{aligned}
&- \int_{\Omega^{\prime}} \int_{\underline{x}_i}^{\overline{x}_i}
\left( \sum_{i=1}^{n}\partial_{x_{i}}\left(p_{i}   J^{\alpha}_{\underline{x}_i}+q_{i} J^{\alpha}_{\overline{x}_i-} \right)\partial_{x_{i}} u \right) u \mathrm{d}x_i\mathrm{d}\hat{x}_i
\\
=&
-\int_{\Omega^{\prime}}
\left( \left. \sum_{i=1}^{n}\left(p_{i}   J^{\alpha}_{\underline{x}_i}+q_{i} J^{\alpha}_{\overline{x}_i-} \right)\partial_{x_{i}} u \right|_{x_i=\underline{x}_i}^{x_i=\overline{x}_i} \right) u \mathrm{d}\hat{x}_i + \int_{\Omega^{\prime}} \int_{\underline{x}_i}^{\overline{x}_i}
\left( \sum_{i=1}^{n}\left(p_{i}   J^{\alpha}_{\underline{x}_i}+q_{i} J^{\alpha}_{\overline{x}_i-} \right)\partial_{x_{i}} u \right)\partial_{x_{i}} u \mathrm{d}x_i\mathrm{d}\hat{x}_i
\\
=& \int_{\Omega^{\prime}} \int_{\underline{x}_i}^{\overline{x}_i}
\left( \sum_{i=1}^{n}\left(p_{i}   J^{\alpha}_{\underline{x}_i}+q_{i} J^{\alpha}_{\overline{x}_i-} \right)\partial_{x_{i}} u \right)\partial_{x_{i}} u \mathrm{d}x_i\mathrm{d}\hat{x}_i \\ 
=&  \int_{\Omega}
fu\mathrm{d}\boldsymbol{x}.
\end{aligned}
$$

If we denote $w:=J^{\alpha}_{\underline{x}_i} \partial_{x_i} u $, we can get $\partial_{x_i} u = \partial_{x_i}^{\alpha}w$. Following the assumption $u\in \hat{H}_1(\Omega)$, we have $w(\cdot, \hat{x}_i)\in H_{\alpha}(\underline{x}_i,\overline{x}_i)$ for fixed $\hat{x}_i$. Then by \cite{KubRys}
$$
\int_{\Omega^{\prime}} \int_{\underline{x}_i}^{\overline{x}_i} p_i J^{\alpha}_{\underline{x}_i} (\partial_{x_i} u)  (\partial_{x_i} u)
\mathrm{d}x_i\mathrm{d}\hat{x}_i=\int_{\Omega^{\prime}} \int_{\underline{x}_i}^{\overline{x}_i} p_i w \cdot \partial_{x_i}^{\alpha} w
\mathrm{d}x_i\mathrm{d}\hat{x}_i \geq C \int_{\Omega^{\prime}} p_i \|w(\cdot, \hat{x}_i)\|^2_{L^2(\underline{x}_i,\overline{x}_i)} \mathrm{d}\hat{x}_i.
$$
By the definition  \ref{Def:R-L_int}, $J^{\alpha}_{\overline{x}_i-}$ is the adjoint of $J^{\alpha}_{\underline{x}_i}$, therefore we get
\begin{equation}\label{uniqu_1}
\begin{aligned}
& \int_{\Omega^{\prime}} \int_{\underline{x}_i}^{\overline{x}_i}
\left( \sum_{i=1}^{n}\left(p_{i}   J^{\alpha}_{\underline{x}_i}+q_{i} J^{\alpha}_{\overline{x}_i-} \right)\partial_{x_{i}} u \right)\partial_{x_{i}} u \mathrm{d}x_i\mathrm{d}\hat{x}_i \\
\geq & C \int_{\Omega^{\prime}} \sum^{n}_{i=1} \left( p_i \|w(\cdot, \hat{x}_i)\|^2_{L^2(\underline{x}_i,\overline{x}_i)} + q_i \|w(\cdot, \hat{x}_i)\|^2_{L^2(\underline{x}_i,\overline{x}_i)} \right)
\mathrm{d}\hat{x}_i
\\
= & C \sum_{i=1}^{n}\int_{\Omega^{\prime}} \|w(\cdot, \hat{x}_i)\|^2_{L^2(\underline{x}_i,\overline{x}_i)}\mathrm{d}\hat{x}_i.
\end{aligned}
\end{equation}
Hereafter, we use the notation $C>0$ to represent generic constants that are independent of the functions being considered but dependent on parameters such as $\alpha$ and $\Omega$.

On the other hand, we have 
\begin{equation}\label{uniqu_2}
\int_{x^{\prime}} \int_{\underline{x}_i}^{\overline{x}_i}
\left( \sum_{i=1}^{n}\left(p_{i}   J^{\alpha}_{\underline{x}_i}+q_{i} J^{\alpha}_{\overline{x}_i-} \right)\partial_{x_{i}} u \right)\partial_{x_{i}} u \mathrm{d}x_i\mathrm{d}\hat{x}_i =  \int_{\Omega}
f(\boldsymbol{x})u(\boldsymbol{x})\mathrm{d}\boldsymbol{x}
\leq \|u\|_{L^2(\Omega)}  \|f\|_{L^2(\Omega)}.
\end{equation}
By Proposition \ref{prop:norm_ine}, we obtain
\begin{equation}\label{uniqu_3}
\begin{aligned}
 \|J^{\alpha}_{\underline{x}_i} \partial_{x_i} u\|^4_{L^2(\Omega)}=&\left(\int_{\Omega^{\prime}} \|J^{\alpha}_{\underline{x}_i} \partial_{x_i} u(\cdot, \hat{x}_i)\|^2_{L^2(\underline{x}_i,\overline{x}_i)}\mathrm{d}\hat{x}_i\right)^2 \\
= & \|J^{\alpha}_{\underline{x}_i} \partial_{x_i} u\|^2_{L^2(\Omega)} \left(\int_{\Omega^{\prime}} \|J^{\alpha}_{\underline{x}_i} \partial_{x_i} u(\cdot, \hat{x}_i)\|^2_{L^2(\underline{x}_i,\overline{x}_i)}\mathrm{d}\hat{x}_i \right) \\
\geq & C \| J^{\alpha}_{\underline{x}_i} \partial_{x_i} u\|^2_{L^2(\Omega)} \left( \int_{\Omega^{\prime}} \| u(\cdot, \hat{x}_i)\|^2_{L^2(\underline{x}_i,\overline{x}_i)}\mathrm{d}\hat{x}_i \right) \\
= & C \|J^{\alpha}_{\underline{x}_i} \partial_{x_i} u\|^2_{L^2(\Omega)} \| u\|^2_{L^2(\Omega)} .
\end{aligned}
\end{equation}
Combining \eqref{uniqu_1} - \eqref{uniqu_3}, we find 
$$
C\left( \sum_{i=1}^{n} \|J^{\alpha}_{\underline{x}_i} \partial_{x_i} u\|^2_{L^2(\Omega)}
 \right) \| u\|^2_{L^2(\Omega)} \leq \|f\|^2_{L^2(\Omega)}\|u\|^2_{L^2(\Omega)}.
$$
If $\|u\|_{L^2(\Omega)}\neq 0$, then we proof
$$ \| u\|_{\hat{H}_{1-\alpha}(\Omega)} \leq C \|f\|_{L^2(\Omega)}.
$$
Thus, we prove \eqref{uniq_ine}. As for the inhomogeneous boundary condition, if there exists two different solutions $u_1$ and $u_2$. We can define $u:=u_1-u_2$. Since the equation is linear, $u$ is the solution for the homogeneous situation. The conclusion is consistent. 
\end{proof}

\section{Fractional weak adversarial networks (f-WANs) framework}\label{Sect:WANf}
 
In this section, we derive the weak formulation of the fractional order equation. Building upon this formulation, we propose a novel weak adversarial network for solving the equation.

\subsection{The weak formulation of the model problem}\label{weak_integrate}

In the case when the source term $f$ in \eqref{frac_prob}-\eqref{frac_oper} is not smooth, the resulting solution $u$ may not belong to the set of $C^2(\Omega)$. This observation serves as the motivation to consider its weak formulation. In view of the assumption that $p_i,q_i$ are all constants, we set $p_i=q_i=\frac{1}{2n}$ for simplicity.
The fractional equation \eqref{frac_prob}-\eqref{frac_oper} can be reformulated as followings 
\begin{equation}\label{strong_form_int}
\begin{aligned}
&-\sum^{n}_{i=1} \frac{\partial}{\partial x_i} \left(  \int_{\underline{x}_i}^{x_i} \frac{1}{\Gamma(\alpha)} (x_i-w)^{\alpha - 1} \frac{\partial u}{\partial w}(w,\hat{x}_i) \mathrm{d}w \right) \\
&- \sum_{i=1}^{n} \frac{\partial}{\partial x_i} \left(  \int_{x_i}^{\overline{x}_i} \frac{1}{\Gamma(\alpha)} (w-x_i)^{\alpha - 1} \frac{\partial u}{\partial w}(w,\hat{x}_i) \mathrm{d}w \right)  = f, \quad (\hat{x}_i, x_i) \in \Omega ,
\end{aligned}
\end{equation}
with boundary condition 
\begin{equation}\label{strong_form_bound}
u = g(\boldsymbol{x}),  \quad  \text{on} \ \partial \Omega.
\end{equation}

Multiplying a test function
$v\in H^1_0(\Omega)$ on both sides of \eqref{strong_form_int}. Here 
the space $H^1_0(\Omega)$, known as the homogeneous Sobolev space, consists of functions whose weak partial derivatives are integrable in the $L^2$ sense over $\Omega$ with vanishing traces on the boundary $\partial \Omega$. Upon integrating by parts, we can obtain a  weak formula for \eqref{strong_form_int}, which is given by
\begin{equation}\label{weak_form_int}
\begin{aligned}
\langle \mathcal{L}[u],v\rangle :=&\sum_{i=1}^{n} \int_\Omega  \left(  \int_{\underline{x}_i}^{x_i} \frac{1}{\Gamma(\alpha)} (x_i-w)^{\alpha - 1} \frac{\partial u}{\partial w}(w, x^{\prime}) \mathrm{d}w \right) \frac{\partial v}{\partial x_i} (x^{\prime},x_i) \mathrm{d} x_i\mathrm{d} x^{\prime}
\\
& +\sum_{i=1}^{n} \int_\Omega  \left(  \int_{x_i}^{\overline{x}_i} \frac{1}{\Gamma(\alpha)} (w-x_i)^{\alpha - 1} \frac{\partial u}{\partial w}(w,x^{\prime}) \mathrm{d}w \right)  \frac{\partial v}{\partial x_i}(x^{\prime},x_i) \mathrm{d} x_i \mathrm{d} x^{\prime}\\
&-\int_{\Omega} f v \mathrm{d}\boldsymbol{x}  = 0.
\end{aligned}
\end{equation}
Meanwhile, we define the following form corresponding to the Dirichlet boundary condition \eqref{strong_form_bound}:
\begin{equation}\label{weak_form_bound}
\mathcal{B}[u] := \left. (u - g)\right|_{\partial \Omega}.
\end{equation}
We also note that in case the when the boundary condition in \eqref{frac_prob} is given in Neumann type, i.e. $\partial_{n}u=g$ on $\partial \Omega$. Then $\mathcal{B}[u]$ can be defined as
$$
\mathcal{B}[u] := \left. (\frac{\partial u}{\partial \vec{n}} - g)\right|_{\partial \Omega}.
$$

\subsection{Induced operator norm minimization}

The above weak formula \eqref{weak_form_int} can induce an operator norm:

\begin{defn}
We define the operator norm $$\| \mathcal{L}[u]\|_{op} := \max_{v\in H^1_0(\Omega), v\neq 0}  \frac{\left|\langle \mathcal{L}[u],v\rangle\right|}{\|v\|_2}     ,$$ where $\|v\|_2=(\int_{\Omega} |v(x)|^2\mathrm{d}x)^{1/2}. $
\end{defn}

This gives the definition of the operator norm of $\mathcal{L}[u]$ induced from $L^2$ norm. Here the linear functional $\mathcal{L}[u]:H^1_0(\Omega) \mapsto \mathbb{R}$ such that $\mathcal{L}[u](v) \triangleq \langle \mathcal{L}[u],v\rangle$ for fractional equations.

\begin{lemma}
Under the assumption that $u\in \hat{H}_{1}(\Omega)$, satisfies the boundary condition \eqref{weak_form_bound} on $\partial \Omega$, then $u$ is the unique weak solution in $\hat{H}_{1}(\Omega)$ of equation \eqref{strong_form_int} if and only if $\| \mathcal{L}[u]\|_{op}=0$.
\end{lemma}

After applying Theorem \ref{them:uniqu_the}, the subsequent steps of the proof follow a similar approach to the one presented in \cite[Theorem~1]{ZanBao}.

By the above Lemma, since $\| \mathcal{L}[u]\|_{op}\ge 0$ for any $u\in \hat{H}_{1}(\Omega)$, $\|\mathcal{L}[u]\|_{op}$ achieves its minimum over $\hat{H}_1(\Omega)$ when $u$ is the weak solution of \eqref{strong_form_int}.
Based on the above analysis, we can formulate the following minimax problem:
\begin{equation}\label{minimax}
\min_{u\in \hat{H}_{1}} \| \mathcal{L}[u]\|_{op}= \min_{u\in \hat{H}_{1}} \max_{v\in H^1_0} \frac{|\langle \mathcal{L}[u],v\rangle|^2}{\|v\|^2_2}.
\end{equation}

In the next section, we will propose a neural network to find the optimal solution $u$ for the minimax problem \eqref{minimax}.

\subsection{Weak adversarial network framework}\label{Sect:We_ad}
We parameterize \eqref{minimax} using neural networks. Let $u_\theta:\mathbb{R}^d \mapsto \mathbb{R}$ and $v_\eta :\mathbb{R}^d \mapsto \mathbb{R}$ denote the parameterization of $u$ and $v$, respectively, where $\theta$ and $\eta$ are the trainable model weights. Then we can express the minimax problem as follows:
\begin{equation}\label{minimax_para}
\min_{u_\theta \in \hat{H}_{1,1}} \max_{v_\eta \in H^1_0} \frac{|\langle \mathcal{L}[u_\theta],v_\eta \rangle|^2}{\|v_\eta\|^2_2} .
\end{equation}

During the learning process of the networks, we first fix $\eta$ and optimize $\theta$ to minimize \eqref{minimax_para}. Once we acquire the optimal $\theta$, we fix it and optimize $\eta$ to challenge $\theta$ and maximize \eqref{minimax_para}. The neural network approximated solution is obtained after steps of loop iterations.
The schematic of the WAN methods is shown in Figure \ref{wan}.
\begin{figure}[!hbt]
	\centering
	\centering
	\includegraphics[width=4in]{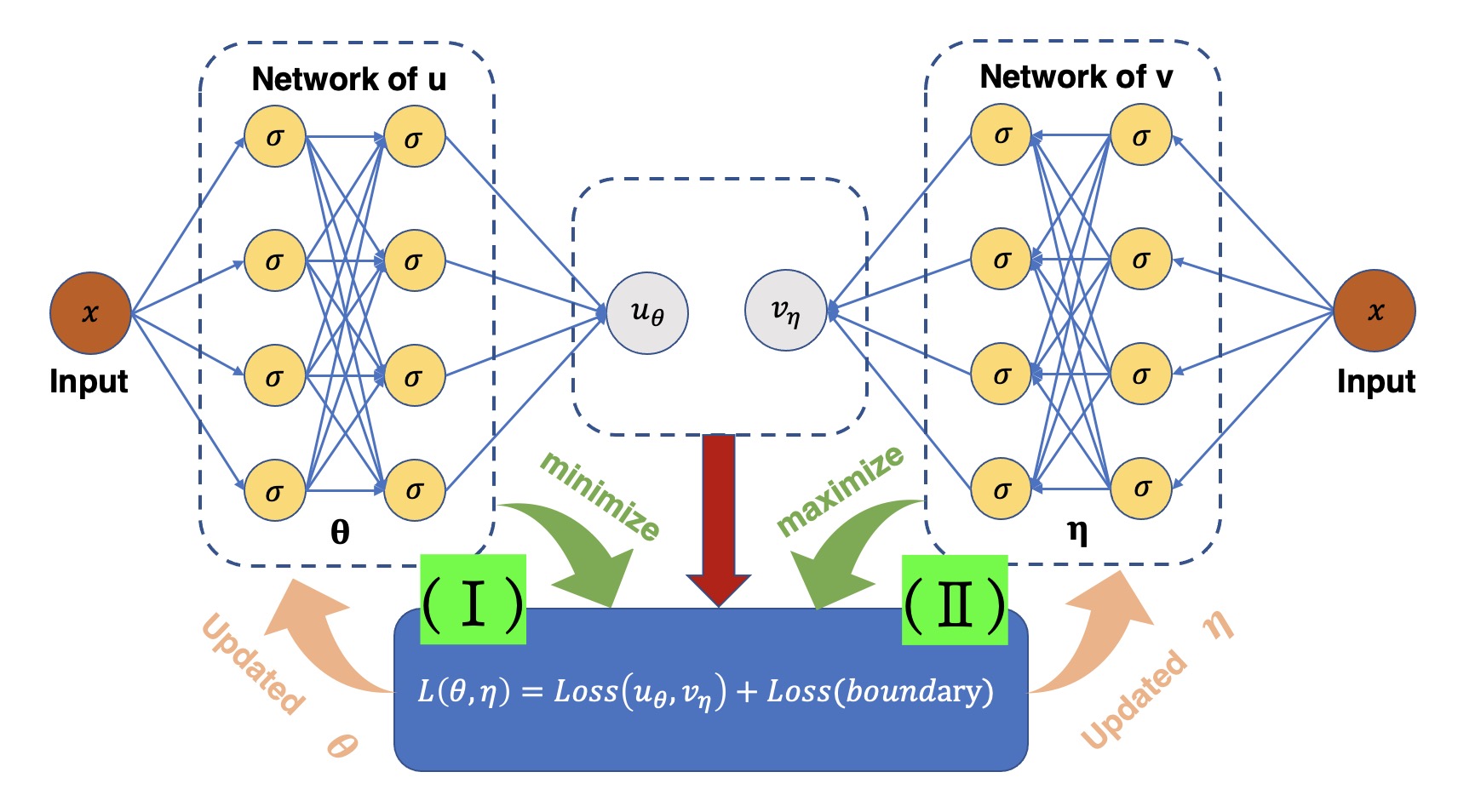}
	\caption{Schematic of the WAN for solving fractional partial differential equations.}
	\label{wan}
\end{figure}
In the interior of $\Omega$, the objective function of $u_\theta$ and $v_\eta$ is $$
L_{\mathrm{int}}(\theta, \eta) \triangleq  |\langle \mathcal{L}[u_\theta],v_\eta\rangle|^2 / \|v_\eta\|^2_2.
$$
In the meantime, the weak solution on the boundary $\partial \Omega$ must satisfy the boundary condition \eqref{weak_form_bound}. Therefore, the objective function is given by:
$$
L_{\mathrm{bound}}(\theta) \triangleq \left|u_{\theta}-g\right|^2.
$$
To sum up, we combine the two objective functions of the interior and the boundary to obtain the total objective function, given by:
\begin{equation}\label{minimax_objective}
\min_\theta \max_\eta L(\theta, \eta),\quad \text{where}\quad L(\theta, \eta) \triangleq L_{\mathrm{int}}(\theta, \eta) + \beta L_{\mathrm{bound}}(\theta),
\end{equation}
with $\beta$  is a regularization parameter that balances the relative importance of the interior and boundary terms.

\subsection{Stochastic approximation of operators and training algorithm}\label{Sect:sec_sto_app}

Unlike the outer integral, which is over the whole domain, the inner integral has a changeable upper limit that depends on the value of the outer integral.
In order to discretize the inner integral, we begin by discretizing the outer integral and then generate points within the domain corresponding to the current value of the outer integral using Monte Carlo sampling. Following this, we employ the same method to approximate the inner integral based on the generated points.

To illustrate the steps of the algorithm, without loss of generality, we consider the model problem in two dimensional case. In particular when $n=2$, we write $\mathcal{L} u =f$ as
\begin{equation}\label{frac_adve}
    -\partial_x \left(p_1 J^{\alpha}_{x_l}+q_1 J^{\alpha}_{x_u-} \right) \partial_x u 
-\partial_y \left( p_2 J^{\alpha}_{y_l}+q_2 J^{\alpha}_{y_u-} \right) \partial_y u=f,
\end{equation}
where to simplify the notation we use $x,y$ to denote the variables in the two dimensions and subscript ${\cdot}_l$ and ${\cdot}_u$ for the constants related to the lower and upper boundaries. The parameterized weak formula is given by
\begin{equation}\label{weak_form_network}
\begin{aligned}
\langle \mathcal{L}[u_\theta],v_\eta \rangle =   &\int_{\Omega} \left(  \int_{x_l}^{x} \frac{1}{\Gamma(\alpha)} (x-w)^{\alpha - 1} \frac{\partial u_\theta}{\partial w}(w, y) \mathrm{d}w \right) \frac{\partial v_\eta}{\partial x} (x,y) \mathrm{d}x\mathrm{d}y\\
& + \int_{\Omega} \left(  \int_{x}^{x_u} \frac{1}{\Gamma(\alpha)} (w-x)^{\alpha - 1} \frac{\partial u_\theta}{\partial w}(w, y) \mathrm{d}w \right) \frac{\partial v_\eta}{\partial x} (x,y) \mathrm{d}x\mathrm{d}y \\
&+\int_{\Omega}   \left(  \int_{y_l}^{y} \frac{1}{\Gamma(\alpha)} (y-w)^{\alpha - 1} \frac{\partial u_\theta}{\partial w}(x,w) \mathrm{d}w \right) \frac{\partial v_\eta}{\partial y} (x,y) \mathrm{d}x \mathrm{d}y \\ 
&+\int_{\Omega}   \left(  \int_{y}^{y_u} \frac{1}{\Gamma(\alpha)}(w-y)^{\alpha - 1} \frac{\partial u_\theta}{\partial w}(x,w) \mathrm{d}w \right) \frac{\partial v_\eta}{\partial y} (x,y) \mathrm{d}x \mathrm{d}y \\
&-\int_{\Omega} f v_\eta(x,y) \mathrm{d}x \mathrm{d}y .
\end{aligned}
\end{equation}
Many traditional methods have been proven effective in dealing with singular integrals. In this paper,  in order to  avoid the ``curse of dimensionality",   the Monte Carlo method is employed. 
Traditional meshing requires an exponentially increasing number of nodes as the dimensionality $d$ increases, resulting in high computational costs. However, with the emergence of neural networks, we can conduct high-dimensional numerical experiments at a low cost, even on personal computers.

We denote $\{ (x_i, y_i)  \}_{i=0}^{M_I}$ for the collocation points in the interior domain of $\Omega$ and $\{ (x_i, y_i)  \}_{i=0}^{M_B}$ for the collocation points on the boundary $\partial \Omega$. 
Then the fifth term of the right hand side of the equation (\ref{weak_form_network}) can be handled by
$$
RHS_5 = \int_\Omega f v_\eta(x,y) \mathrm{d}x \mathrm{d}y \approx \frac{1}{M_I} \sum^{M_I}_{i=1} fv_\eta(x_i,y_i).
$$
However the first four remaining terms on the right-hand side require more attention because they involve integrals with limits that depend on the values of the outer integral variables $x$ and $y$.  To be more specific, when the outer layer is discretized into $x_i$ (or $y_i$), the inner upper limit of the integral will change along with $x_i$ (or $y_i$) in the outer layer. 
To differentiate between the discrete $x_i$ (or $y_i$) values in various intervals, we introduce the notation $x_i^l$ (or $y_i^l$) for the lower interval and $x_i^u$ (or $y_i^u$) for the upper interval.
For the sake of convenience, we will only consider the first and second terms of equation (\ref{weak_form_network}):
\begin{equation}\label{RHS_1_2}
\begin{aligned}
RHS_{1,2} = & \int_\Omega   \left(  \int_{x_l}^{x} \frac{1}{\Gamma(\alpha)} (x-w)^{\alpha - 1} \frac{\partial u_\theta}{\partial w}(w, y) \mathrm{d}w \right) \frac{\partial v_\eta}{\partial x} (x,y) \mathrm{d}x\mathrm{d}y \\
&+ \int_\Omega   \left(  \int_x^{x_u} \frac{1}{\Gamma(\alpha)} (w-x)^{\alpha - 1} \frac{\partial u_\theta}{\partial w}(w, y) \mathrm{d}w \right) \frac{\partial v_\eta}{\partial x} (x,y) \mathrm{d}x\mathrm{d}y \\ 
\approx& \sum_{i=1}^{M_I} \frac{1}{M_I}\left( \int_{x_l}^{x_i^l} \frac{1}{\Gamma(\alpha)} (x_i^l-w)^{\alpha - 1} \frac{\partial u_\theta}{\partial w}(w, y_i^l)\mathrm{d}w \right) \frac{\partial v_\eta}{\partial x} (x_i^l,y_i^l) \\
&+ \sum_{i=1}^{M_I} \frac{1}{M_I}\left( \int_{x_i^u}^{x_u} \frac{1}{\Gamma(\alpha)} (w-x_i^u)^{\alpha - 1} \frac{\partial u_\theta}{\partial w}(w,y_i^u) \mathrm{d}w \right) \frac{\partial v_\eta}{\partial x} (x_i^u,y_i^u) \\
\approx & \sum_{i=1}^{M_I}  \sum_{j=1}^{N} \frac{1}{M_I \cdot N} \frac{1}{\Gamma(\alpha)} (x_i^l-w_j^{x_l})^{\alpha - 1} \frac{\partial u_\theta}{\partial w}(w_j^{x_l},y_i^l) \frac{\partial v_\eta}{\partial x} (x_i^l,y_i^l) \\
& + \sum_{i=1}^{M_I}  \sum_{j=1}^{N} \frac{1}{M_I \cdot N} \frac{1}{\Gamma(\alpha)} (w_j^{x_u}-x_i^u)^{\alpha - 1} \frac{\partial u_\theta}{\partial w}(w_j^{x_u}, y_i^u) \frac{\partial v_\eta}{\partial x} (x_i^u,y_i^u),
\end{aligned}
\end{equation}
here $w_j^{x_l}$ (or $ w_j^{x_u}), j=1,\cdots, N$ means the collocation points in the range of $[x_l,x_i^l]$ (or $[x_i^u,x_u])$ for $i=1,\cdots,M_I$.
Finally, we arrive at the following approximation for the weak form:
$$
\begin{aligned}
\langle \mathcal{L}[u_\theta],v_\eta \rangle \approx &
\sum_{i=1}^{M_I}  \sum_{j=1}^{N} \frac{1}{M_I \cdot N} \frac{1}{\Gamma(\alpha)} \left( (x_i^l-w_j^{x_l})^{\alpha - 1} \frac{\partial u_\theta}{\partial w}(w_j^{x_l},y_i^l) \frac{\partial v_\eta}{\partial x} (x_i^l,y_i^l) \right. \\ &
\left. + (w_j^{x_u}-x_i^u)^{\alpha - 1} \frac{\partial u_\theta}{\partial w}(w_j^{x_u}, y_i^u) \frac{\partial v_\eta}{\partial x} (x_i^u,y_i^u)
\right)
 \\
& + \sum_{i=1}^{M_I}  \sum_{j=1}^{N} \frac{1}{M_I \cdot N} \frac{1}{\Gamma(\alpha)} \left( (y_i^l-w_j^{y_l})^{\alpha - 1} \frac{\partial u_\theta}{\partial w}(x_i^l,w_j^{y_l}) \frac{\partial v_\eta}{\partial x} (x_i^l,y_i^l) \right. \\ &
\left. + (w_j^{y_u}-y_i^u)^{\alpha - 1} \frac{\partial u_\theta}{\partial w}(x_i^u,w_j^{y_u}) \frac{\partial v_\eta}{\partial x} (x_i^u,y_i^u)
\right) - \frac{1}{M_I} \sum^{M_I}_{i=1} fv_\eta(x_i,y_i).
\end{aligned}
$$
And the loss function on the boundary $\partial \Omega$ is defined by 
$$
L_{\mathrm{bound}} = \sum_{i=1}^{M_B} \frac{1}{M_B} \left|u_{\theta}(x_i, y_i)-g(x_i,y_i)\right|^2. 
$$ 
Based on the stochastic approximations of the interior and boundary objective functions discussed above, we can ultimately obtain the total objective function
\begin{equation}\label{total_objective}
L(\theta,\eta) := L_{int} + \beta L_{bound} = |\langle \mathcal{L}[u_\theta],v_\eta \rangle| / \|v_\eta\|_2^2 + \beta \sum_{i=1}^{M_B} \frac{1}{M_B} \left|u_{\theta}(x_i,y_i)-g(x_i,y_i)\right|^2.    
\end{equation}
While looking for the saddle point of formula \eqref{total_objective}, we use TensorFlow \cite{AbaBar} to automatically calculate $\nabla_\theta L(\theta,\eta)$ and $\nabla_\eta L(\theta,\eta)$. The resulting algorithm is outline in Algorithm \ref{Algo:alg_wan}.

\begin{algorithm}[!hbt]
\caption{Weak adversarial network (WAN) for solving fractional differential equations} 

\begin{flushleft}
{\bf Input:} 

$\Omega:$ domain; 
$\alpha:$ fractional order; 
$M_I/M_B:$ number of collocation points in the domain or on the boundary; 
$N:$ number of collocation points in the integrals in fractional equation; 
$K_u/K_v:$ number of solution/adversarial network parameter updates per iteration;
$\tau_\theta:$ learning rate for network parameter $\theta$ of weak solution $u_\theta$;
$\tau_\eta:$ learning rate for network parameter $\eta$ of test function $v_\eta$.

{\bf Initialize:} parameters $\theta,\eta$ in network architecture $u_{\theta}$ and $v_{\eta}$.
\end{flushleft}
\begin{algorithmic}
\WHILE{not converged}
\STATE Sample collocation points 

$\{ (x_i^l,y_i^l) \in [x_l,x_i^l]: i \in [M_I] \}$, $\{ w_j^{x_l} \in [x_l,x_i^l]: j \in [N] \}$, 

$\{ (x_i^u,y_i^u) \in [x_i^u,x_u]: i \in [M_I] \}$, $\{ w_j^{x_u} \in [x_i^u,x_u]: j \in [N] \}$, 

$\{ (x_i^l,y_i^l) \in [y_l,y_i^l]: i \in [M_I] \}$, $\{ w_j^{y_l} \in [y_l,y_i^l]: j \in [N] \}$, 

$\{ (x_i^u,y_i^u) \in [y_l,y_i^l]: i \in [M_I] \}$, $\{ w_j^{y_l} \in [y_i^u,y_u]: j \in [N] \}$, 

and
$\{ (x_i,y_i) \in \partial \Omega: i \in [M_B] \}$ .
\STATE Update weak solution network parameter
	\FOR{$k=1,\cdots,K_u$ }
		\STATE  Using automatic  differential technique to calculate $\nabla_\theta L$;
		\STATE Update $\theta \leftarrow \theta - \tau_\theta \nabla_\theta L$.
	\ENDFOR
\STATE Update test function network parameter
	\FOR{$k=1,\cdots,K_v$ }
		\STATE  Using automatic differential technique to calculate $\nabla_\eta L$;
		\STATE Update $\eta \leftarrow \eta + \tau_\eta \nabla_\eta L$.
	\ENDFOR
\ENDWHILE
\end{algorithmic}
\begin{flushleft}
{\bf Output:} The weak solution $u_\theta$.
\end{flushleft}
\label{Algo:alg_wan}
\end{algorithm}

\section{Numerical experiments} \label{Sect:NumerExper}

\subsection{Experiment setup}

Recall that we have introduced two networks in Section \ref{Sect:We_ad} to approximate the weak solution $u$ and test function $v$, namely $u_\theta$ and $v_\eta$. 
All neural networks and algorithm presented in this paper are implemented using TensorFlow. Unless stated otherwise, we use Adam \cite{KingBa} as the optimizer with a step size 0.0015 for $u_\theta$ and 0.04 for $v_\eta$. The parameters of the networks are initialized randomly according to TensorFlow's default procedure.

For network $v_\eta$, we adopt a  structure that  consists of 6 hidden layers, each with 50 neurons. The activation function of the first two layers is $\tanh$, while the activation function for even layers is softplus and for odd layers is $\sinh$. The last output layer does not have an activation function. 
The structure of network $u_\theta$ differs slightly form the fully-connected feed forward networks used in most cases. We add a convolutional layer to increase the expressiveness of the neural network and reduce the number of neurons in other layers. Thus, $u_\theta$ has 6 hidden layers, each with 20 neurons. The activation functions of $u_\theta$ are $\tanh$ for the first two layers, and softplus for even layers and $\sinh$ for odd layers. We add a convolutional layer in the fifth layer.

\subsection{Experimental results}

\subsubsection{The solution of the 2-dimensional fractional equation exhibits smoothness}\label{Exm_x2}

We start with fractional equation with $d=2$ in (\ref{strong_form_int}) along with boundary condition (\ref{strong_form_bound}). For convenience, we denote equation (\ref{strong_form_int}) as $\mathcal{D}^\alpha_x u=f$.
where $\Omega=(0,1)^2$.
The exact solution is given by
$u(x,y)=x^2y^2$ and the right-handed side of the equation can be calculated accordingly.
The hyper-parameters are given by $M_I=2500,M_B=400,N=50,K_u=1,K_v=1,\tau_\theta=0.0015,\tau_\eta=0.04,\beta=1000000$.
And in Table \ref{Tab:table_para}, we give  the reference of notations used in the algorithm. 
\begin{table}[]
\caption{List of algorithm parameters.}
\label{Tab:table_para}
\begin{tabular}{c|l}
\hline
\textbf{Notation} & \textbf{Definition}                                                                \\ \hline
$d$               & Dimension of $\Omega \subset \mathbb{R}^d$                                         \\
$M_I$             & Number of collocation points in the interior of $\Omega$                           \\
$M_B$             & Number of collocation points on the boundary $\partial \Omega$                     \\
$N$               & Number of collocation points in the integrals in fractional equation               \\
$\alpha$          & The fractional order                                                               \\
$\beta$           & Weight parameter of boundary loss $L_{bound}$                                      \\
$\tau_\theta$     & Learning rate for network parameter $\theta$ of weak solution $u_\theta$           \\
$\tau_\eta$      & Learning rate for network parameter $\eta$ of test function $v_\eta$               \\
$K_u$             & Number of inner iteration that solution network parameter updates per iteration    \\
$K_v$             & Number of inner iteration that adversarial network parameter updates per iteration \\ \hline
\end{tabular}
\end{table}
In our first numerical experiment, we evaluate the feasibility of the proposed f-WANs method for solving fractional differential equations by applying Algorithm \ref{Algo:alg_wan} for 2000 iterations with different fractional order values $\alpha$. The true solution of $u$ is shown in Figure \ref{2d_true}. Specifically, we choose $\alpha=0.3,0.6,0.9$  to demonstrate the effectiveness of the proposed method. It is worth noting that for the inner and outer integral formulations in  \eqref{RHS_1_2}, we sample the data points using different methods. For the outer definite integral, we generate data points with a uniform distribution. However, for the inner integral, we divide the interval into $N$ sub-intervals of equal width.
To clarify, let's take the one-dimensional case as an example. We divide $[0,x_i]$ and $[x_i,1]$  into $[0,\frac{1}{N}x_i,\cdots,\frac{N-1}{N}x_i,x_i]$ and $[x_i,\frac{1+(N-1)x_i}{N},\cdots,\frac{N-1+x_i}{N},1]$) for $i \in [1, M_I]$. This approach can alleviate the randomness and instability of the whole neural network system, despite the emergence of singular integrals that may cause larger errors. Nonetheless, the overall relative error is approximately $3\%$, which demonstrates the feasibility of the proposed method.
The trend in Figure \ref{Fig:2d_pred} shows that the error increases as the value of $\alpha$ decreases.  This is consistent with our expectation since smaller   $\alpha$ corresponds to higher singularities in the solution of \eqref{strong_form_int}. Then leading to larger discretization errors. 

 \begin{figure}[H]
	\centering
\includegraphics[width=0.33\textwidth]{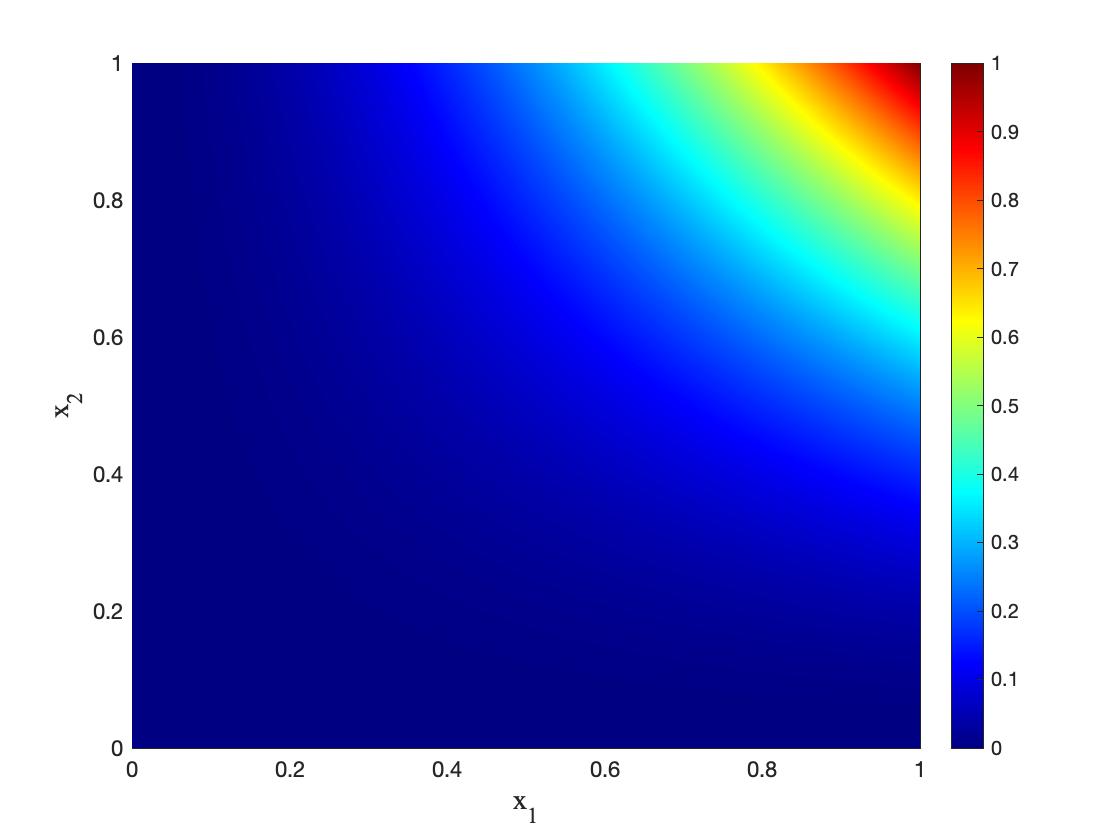}
	\caption{True solution of $u(x,y) = x^2y^2$.}
	\label{2d_true}
\end{figure}

\begin{figure}[H]
\centering
\subfigure[prediction while $\alpha=0.3$]{
\begin{minipage}[t]{0.25\linewidth}
\centering
\includegraphics[width=1.14\textwidth]{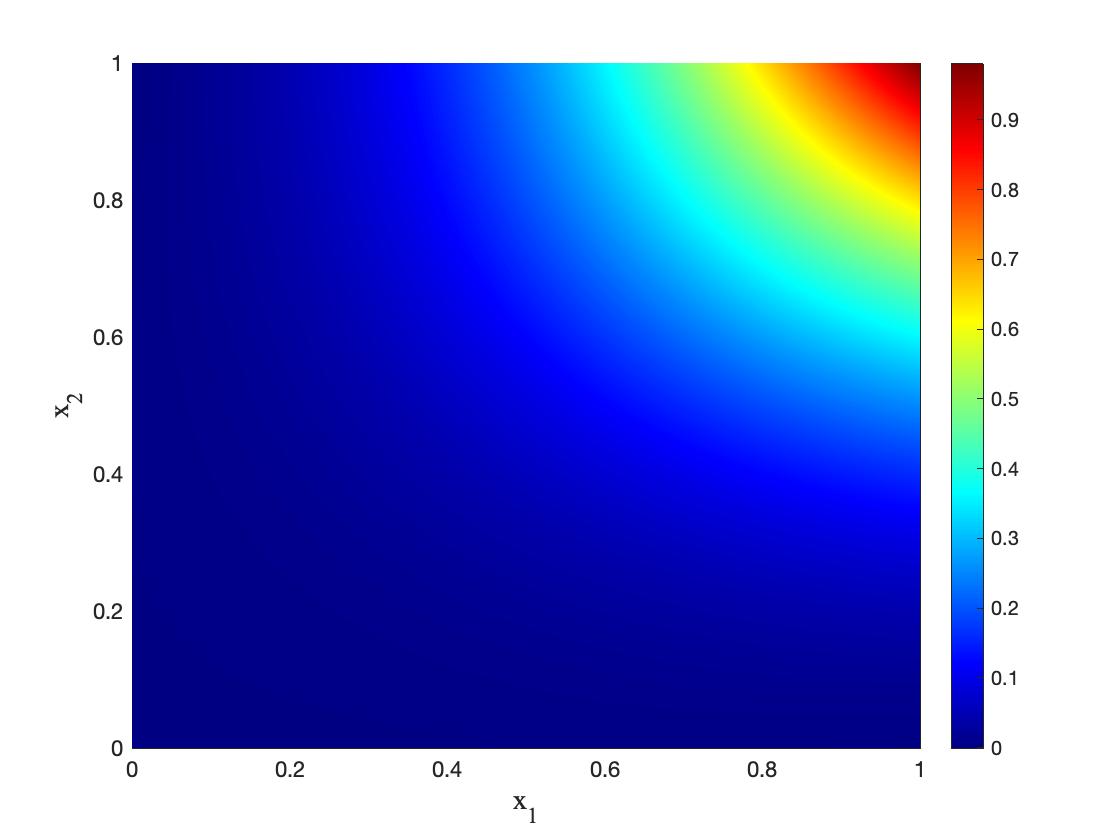}
\end{minipage}
}
\subfigure[prediction while $\alpha=0.6$]{
\begin{minipage}[t]{0.25\linewidth}
\centering
\includegraphics[width=1.14\textwidth]{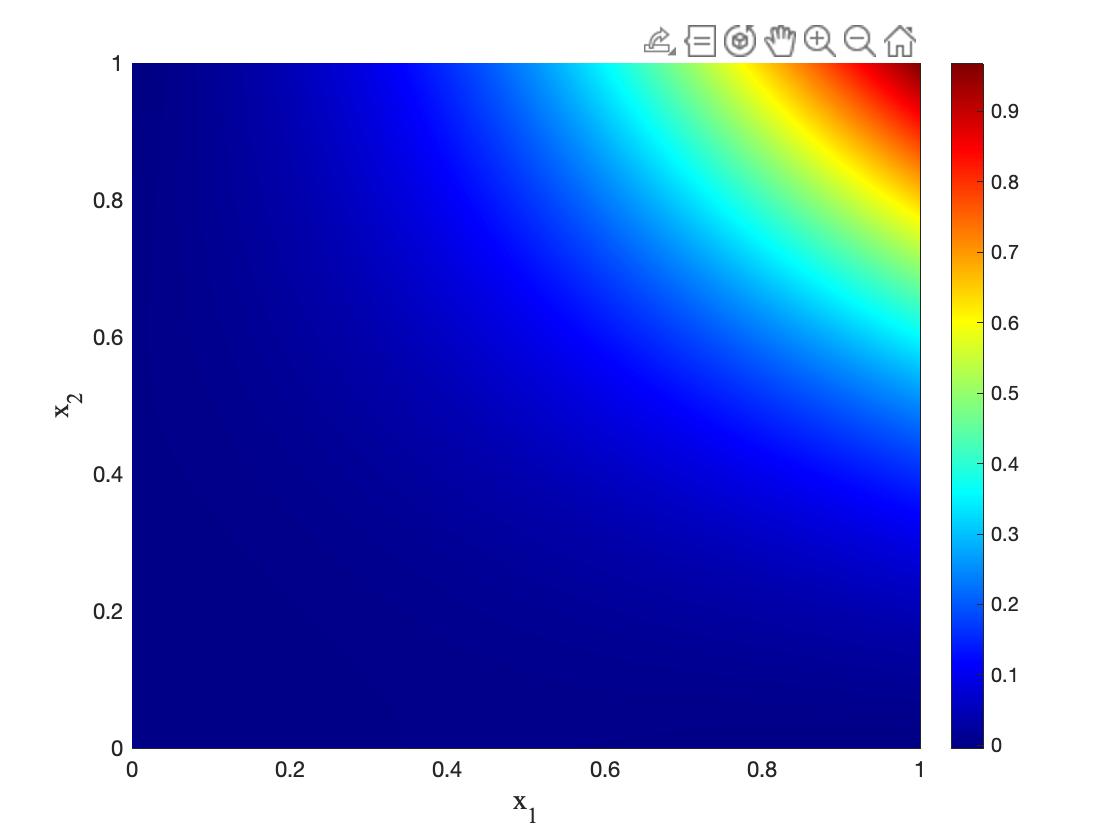}
\end{minipage}
}
\subfigure[prediction while $\alpha=0.9$]{
\begin{minipage}[t]{0.25\linewidth}
\centering
\includegraphics[width=1.14\textwidth]{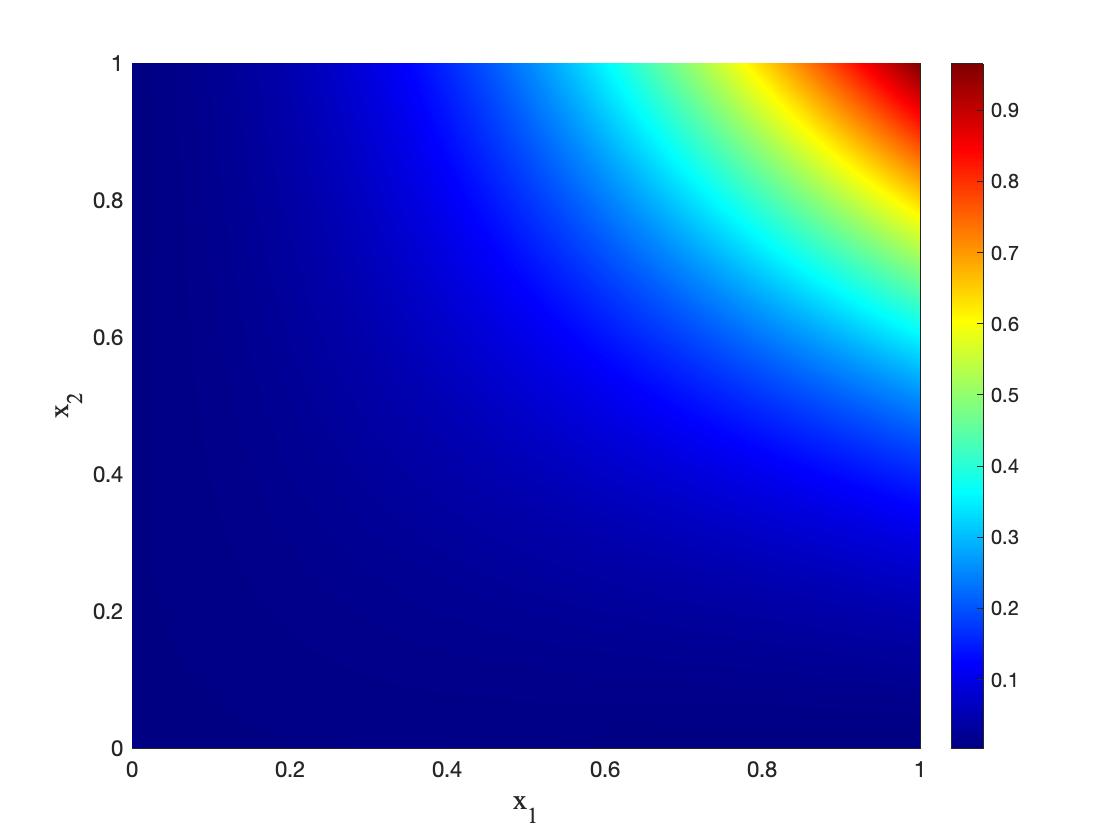}
\end{minipage}
}
                 
\subfigure[difference while $\alpha=0.3$]{
\begin{minipage}[t]{0.25\linewidth}
\centering
\includegraphics[width=1.14\textwidth]{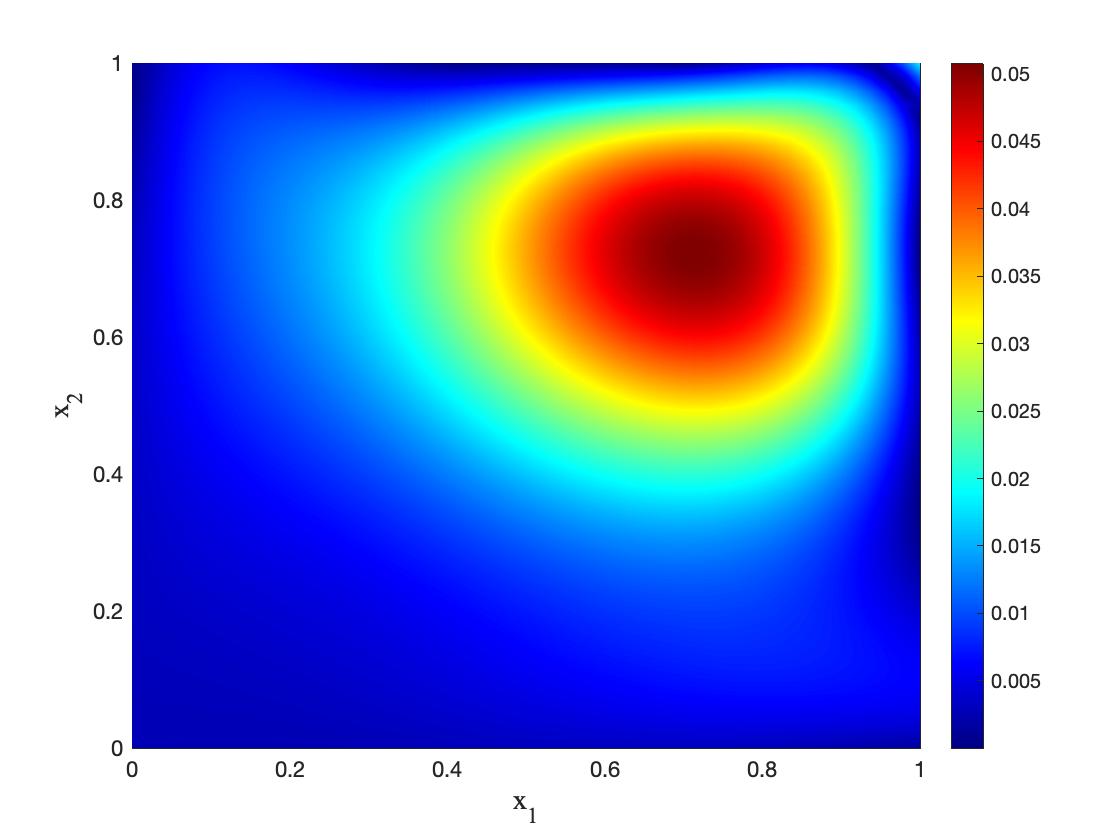}
\end{minipage}
}
\subfigure[difference while $\alpha=0.6$]{
\begin{minipage}[t]{0.25\linewidth}
\centering
\includegraphics[width=1.14\textwidth]{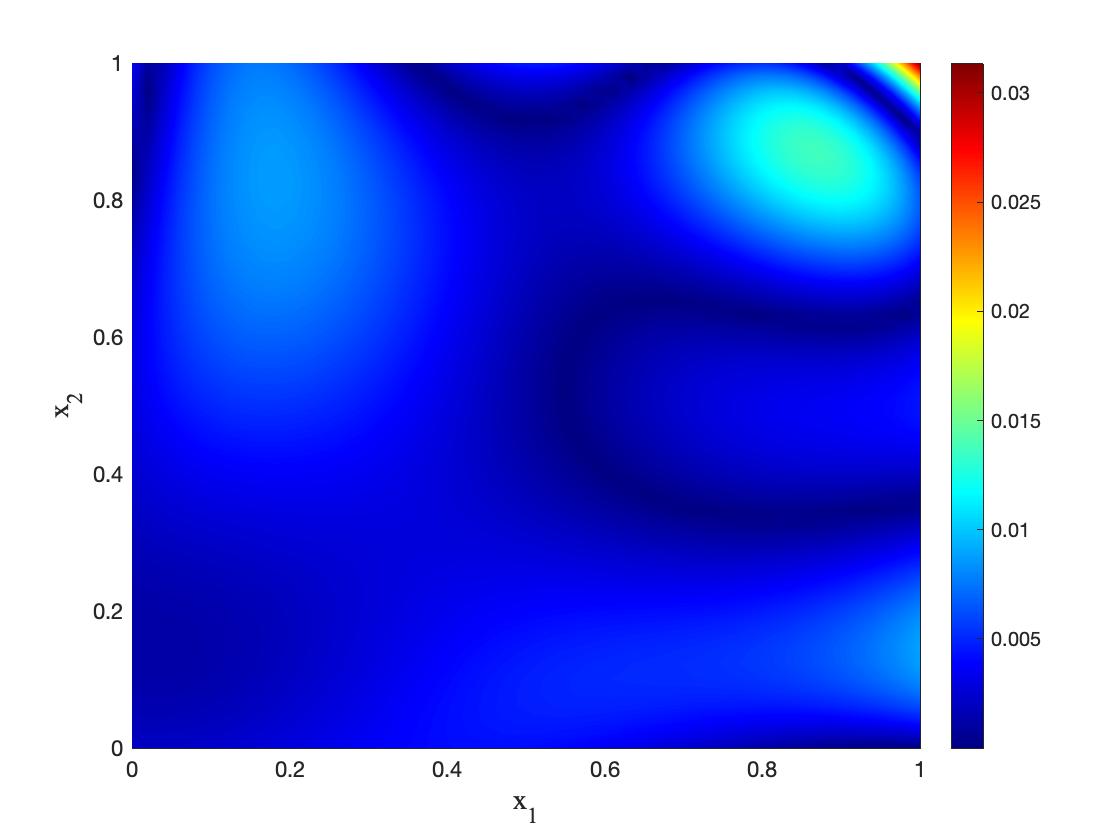}
\end{minipage}
}
\subfigure[difference while $\alpha=0.9$]{
\begin{minipage}[t]{0.25\linewidth}
\centering
\includegraphics[width=1.14\textwidth]{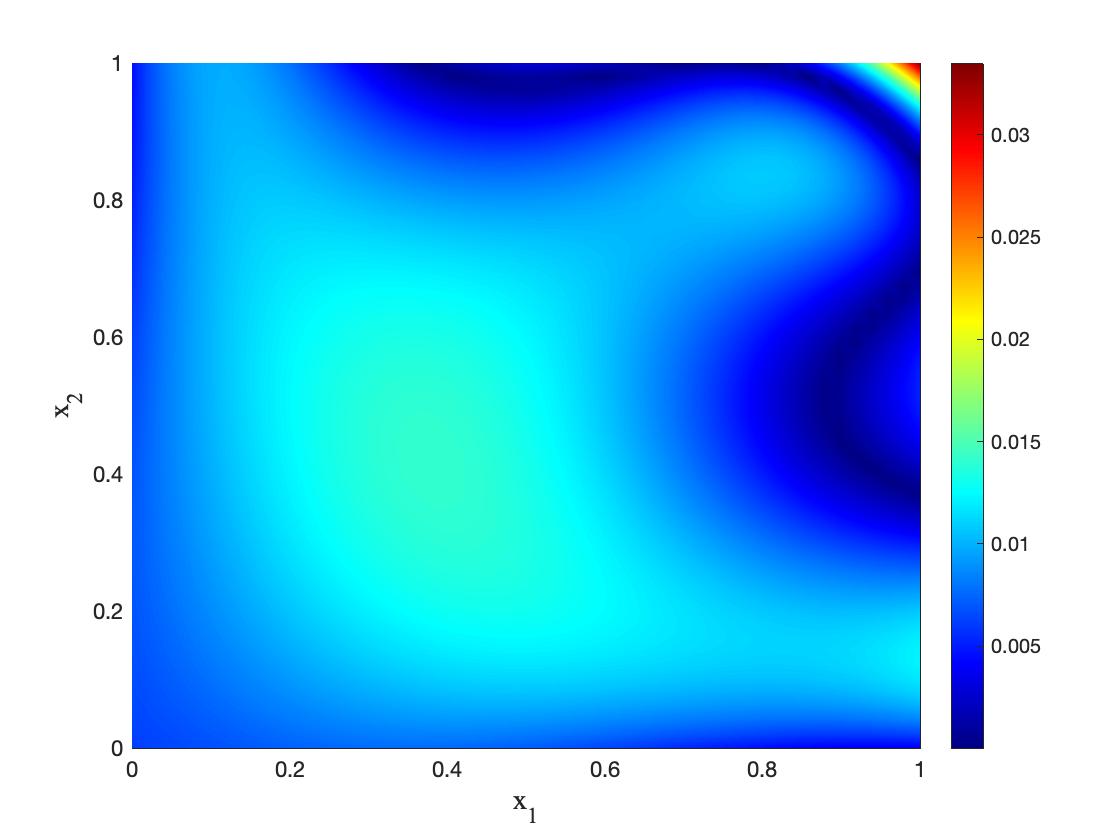}
\end{minipage}
}
\centering
\caption{The first row displays the predicted values of the WAN for $\alpha=0.3,0.6$ and $0.9$ respectively, from left to right. The second row shows the corresponding errors   $|u_{\text{prediction}} - u_{\text{true}}|$.
}
\label{Fig:2d_pred}
\end{figure}

\subsubsection{The solution of the 2-dimensional fractional equation exhibits smoothness with noise}\label{ex:noise}

To investigate the robustness of the WAN with respect to the choice of $\alpha$, Other data is the same as in \ref{Exm_x2} except for a small random perturbation to the Dirichlet value $g$ on the boundary $\partial \Omega$. 
More precisely, we define $g^{\delta}=g+\delta \max\{g\} \xi$,  where $\xi$ is a Gaussian random distribution with zero mean and unit variance. Here the noise level $\delta$ is set to be $5\%$.
To ensure fairness, we set the same random seed to generate $\xi$ and add noise to the Dirichlet values $g$ when solving the equations corresponding to different $\alpha$.
The results are presented in Figure \ref{Fig:2d_pred_with_noise}. In the first row of Figure \ref{Fig:2d_pred_with_noise}, we have highlighted the contour line of 0 in the lower left corner of the image.
As we perturb the Dirichlet value of the boundary, $g$ is not exactly equal to 0 on the left and lower sides of $\partial \Omega$.  As a result, the predictions obtained by the neural network have a small disturbance in the lower left corner. However, they are still very close to 0 in value, and there is no significant difference in magnitude on the colorbar. Furthermore, we have also considered different levels of noise, and the results show that the error increases as the noise level increases.

\begin{figure}[H]
\centering
\subfigure[prediction while $\alpha=0.3$]{
\begin{minipage}[t]{0.25\linewidth}
\centering
\includegraphics[width=1.14\textwidth]{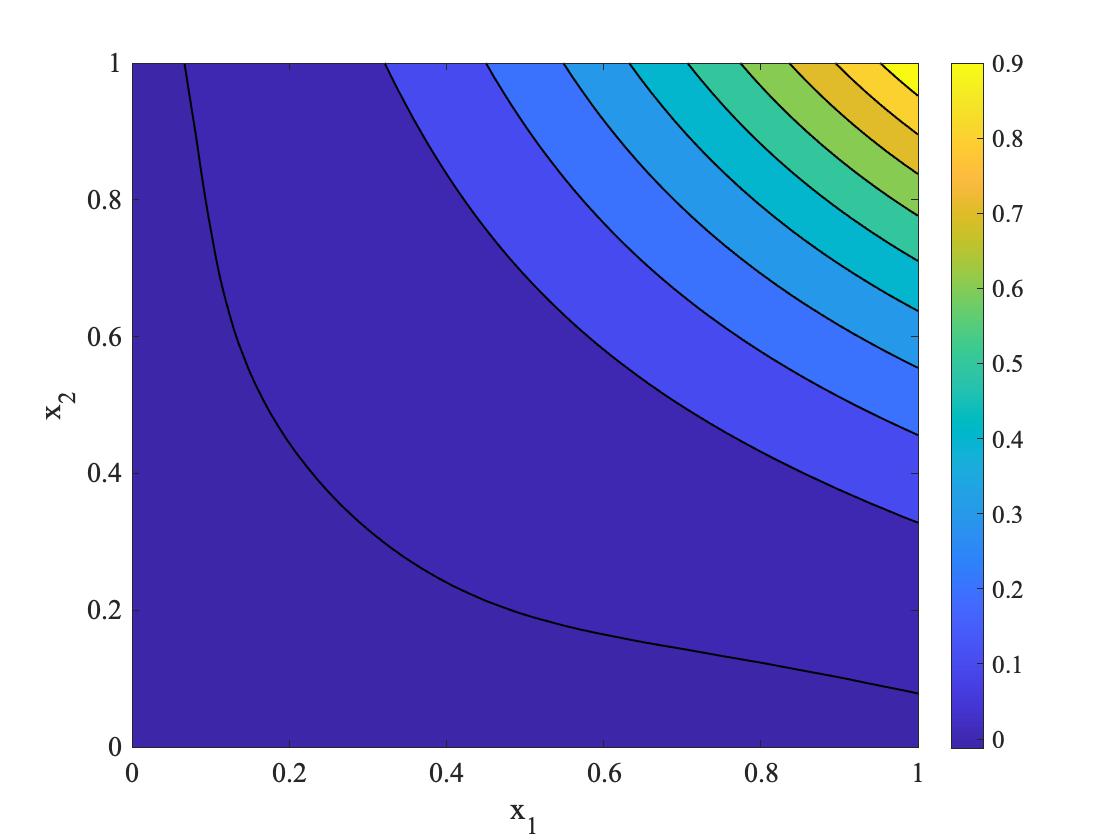}
\end{minipage}
}
\subfigure[prediction while $\alpha=0.6$]{
\begin{minipage}[t]{0.25\linewidth}
\centering
\includegraphics[width=1.14\textwidth]{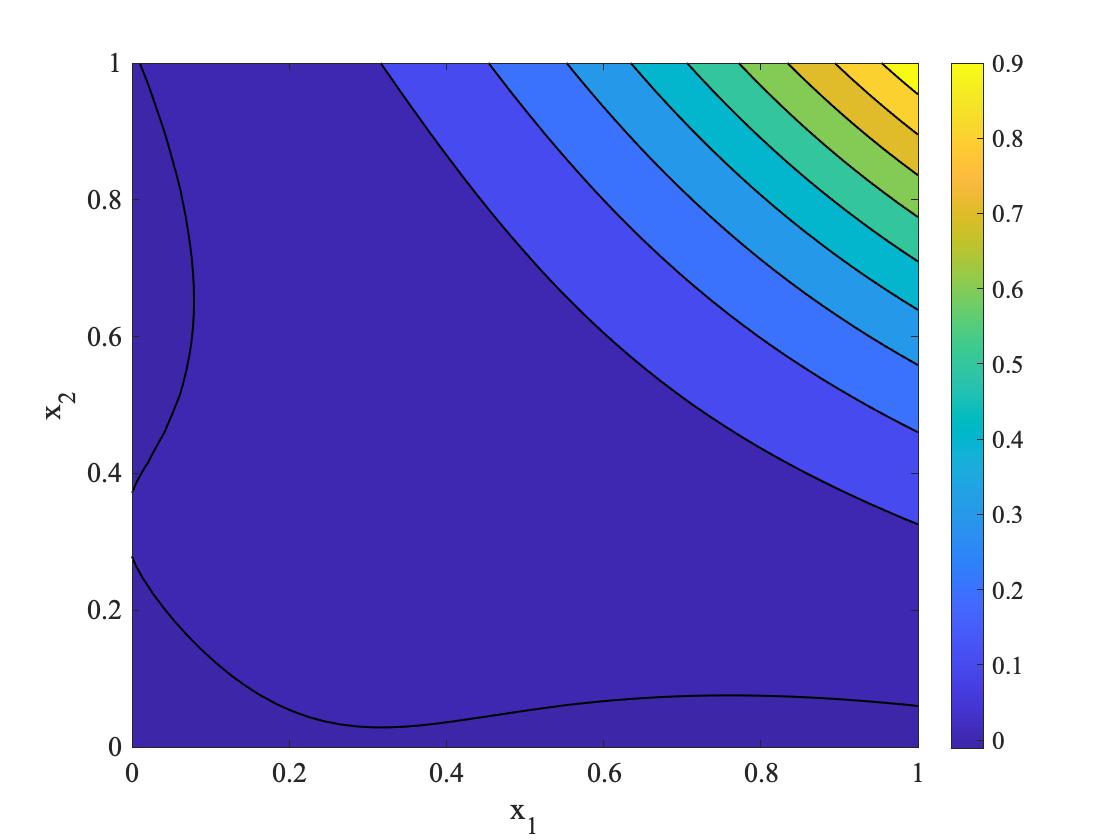}
\end{minipage}
}
\subfigure[prediction while $\alpha=0.9$]{
\begin{minipage}[t]{0.25\linewidth}
\centering
\includegraphics[width=1.14\textwidth]{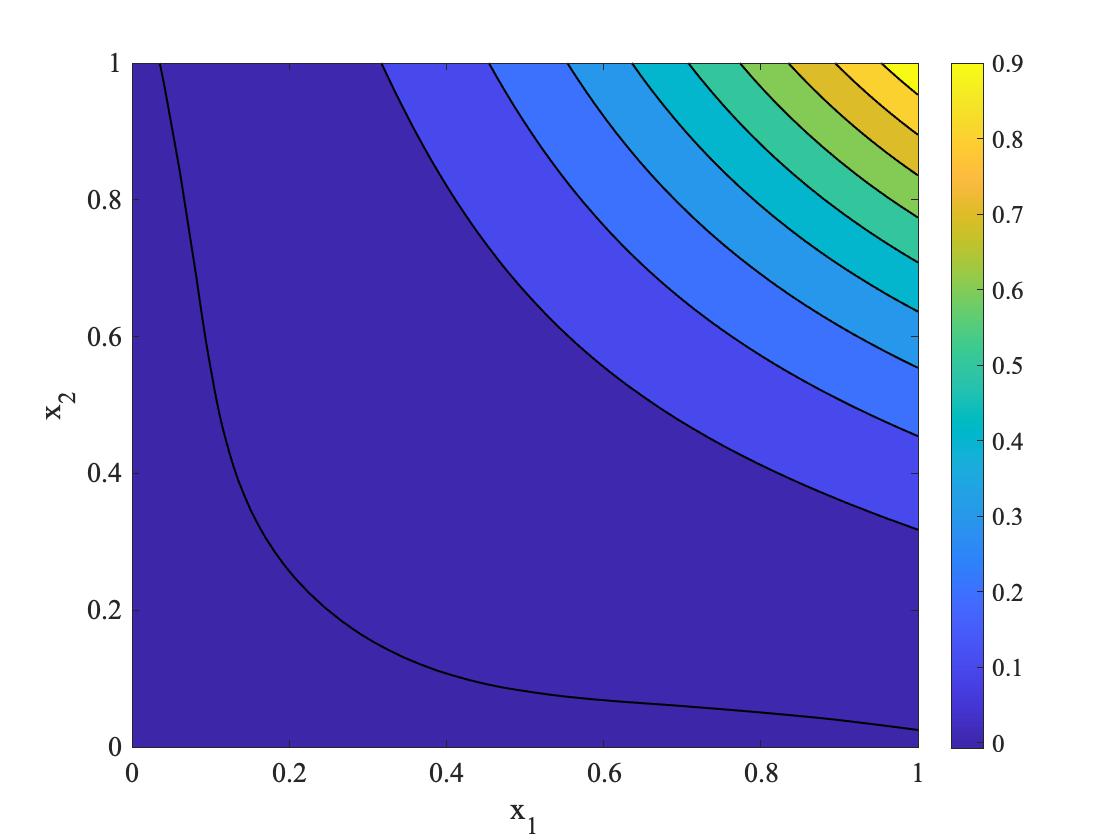}
\end{minipage}
}
                 
\subfigure[difference while $\alpha=0.3$]{
\begin{minipage}[t]{0.25\linewidth}
\centering
\includegraphics[width=1.14\textwidth]{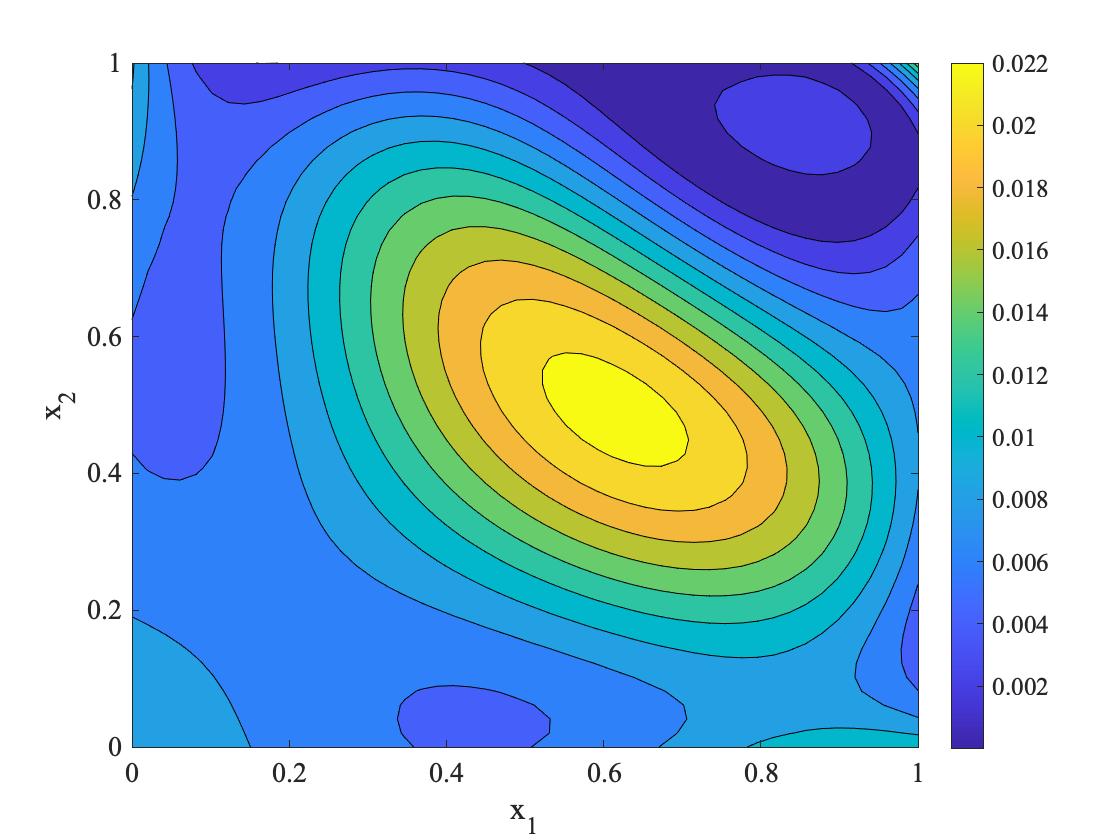}
\end{minipage}
}
\subfigure[difference while $\alpha=0.6$]{
\begin{minipage}[t]{0.25\linewidth}
\centering
\includegraphics[width=1.14\textwidth]{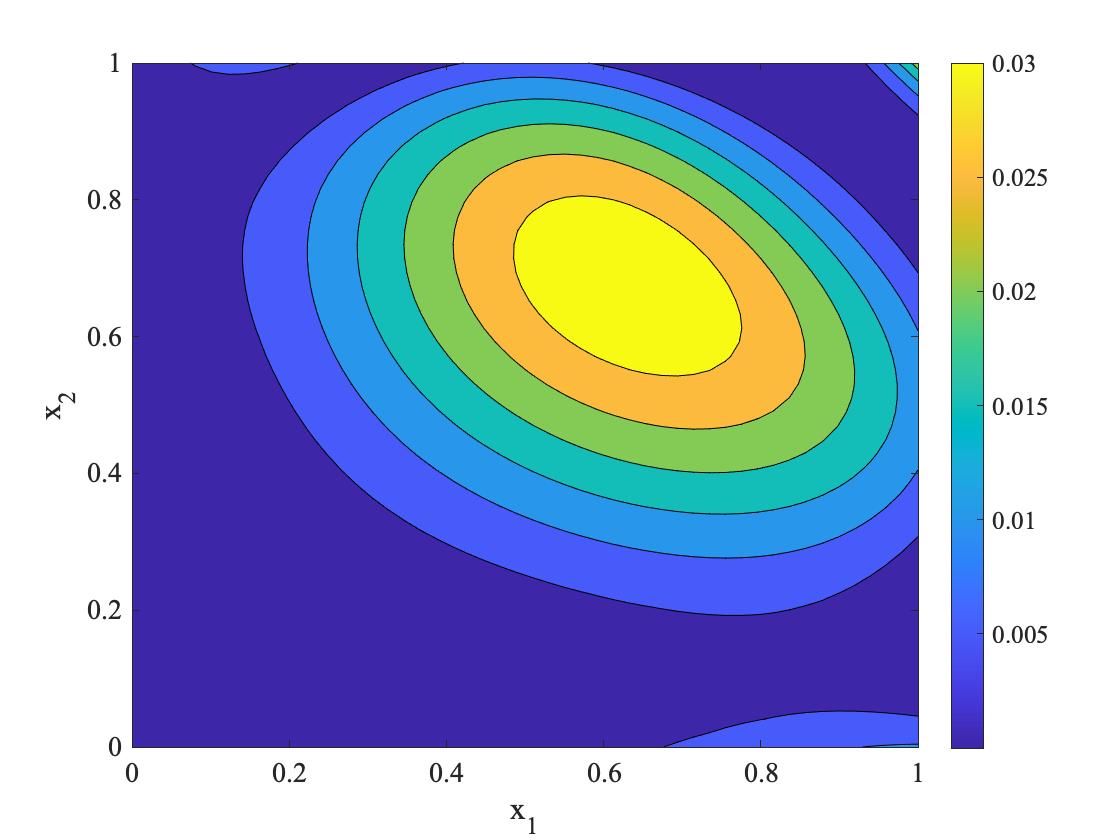}
\end{minipage}
}
\subfigure[difference while $\alpha=0.9$]{
\begin{minipage}[t]{0.25\linewidth}
\centering
\includegraphics[width=1.14\textwidth]{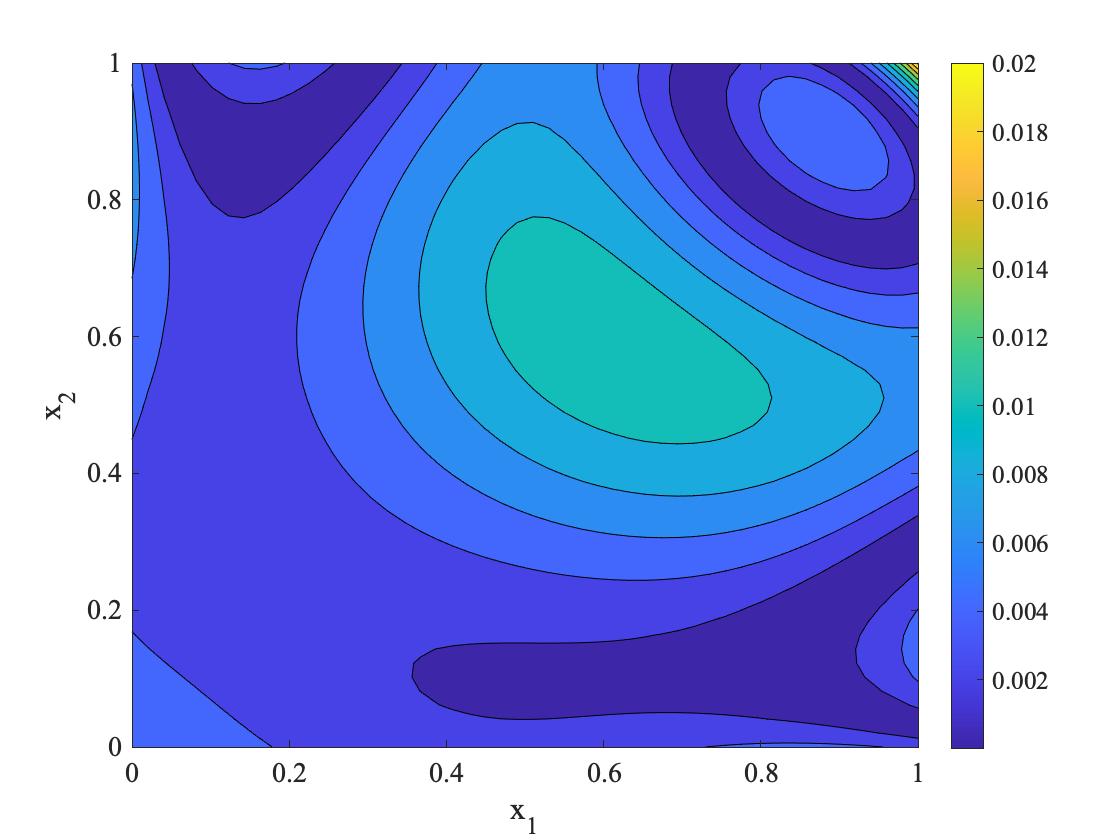}
\end{minipage}
}
\centering
\caption{Results  with different $\alpha$ while noise level chosen by $\delta=5\%$. The first line displays the predictive values of the f-WANs for $\alpha=0.3$, $0.6$, and $0.9$, from left to right. The second row shows the corresponding errors, i.e., the absolute difference between the predicted solution $u^{\delta}_{\text{prediction}}$ and the true solution $u_{\text{true}}$.}
\label{Fig:2d_pred_with_noise}
\end{figure}

\subsubsection{The 2-dimensional fractional equation exhibits a solution with less smoothness}\label{ex:lessSmooth}
In this example, we consider the solution of the model with less smoothness. For instance, the solution of the model problem is chosen by 
$    u(x,y)=x^{\frac{16}{15}}y^{\frac{16}{15}}.$ To better illustrate the results near the point $(0,0)$, where the true solution becomes singular after taking the $2-\alpha$ derivative, we employ a logarithmic scale for both the $x$ and $y$ coordinates.
The true solution is shown in Figure \ref{Fig:2d_true_xy_1516}.
In this example, we set $M_I=2500,M_B=400,N=50,K_u=1,K_v=1,\tau_\theta=0.0001,\tau_\eta=0.001,\beta=1000000$.
We choose $\alpha=0.2,0.4,0.6$, which is different from Section \ref{Exm_x2}, to show our method is still work.
 \begin{figure}[H]
\centering
\includegraphics[width=0.33\textwidth]{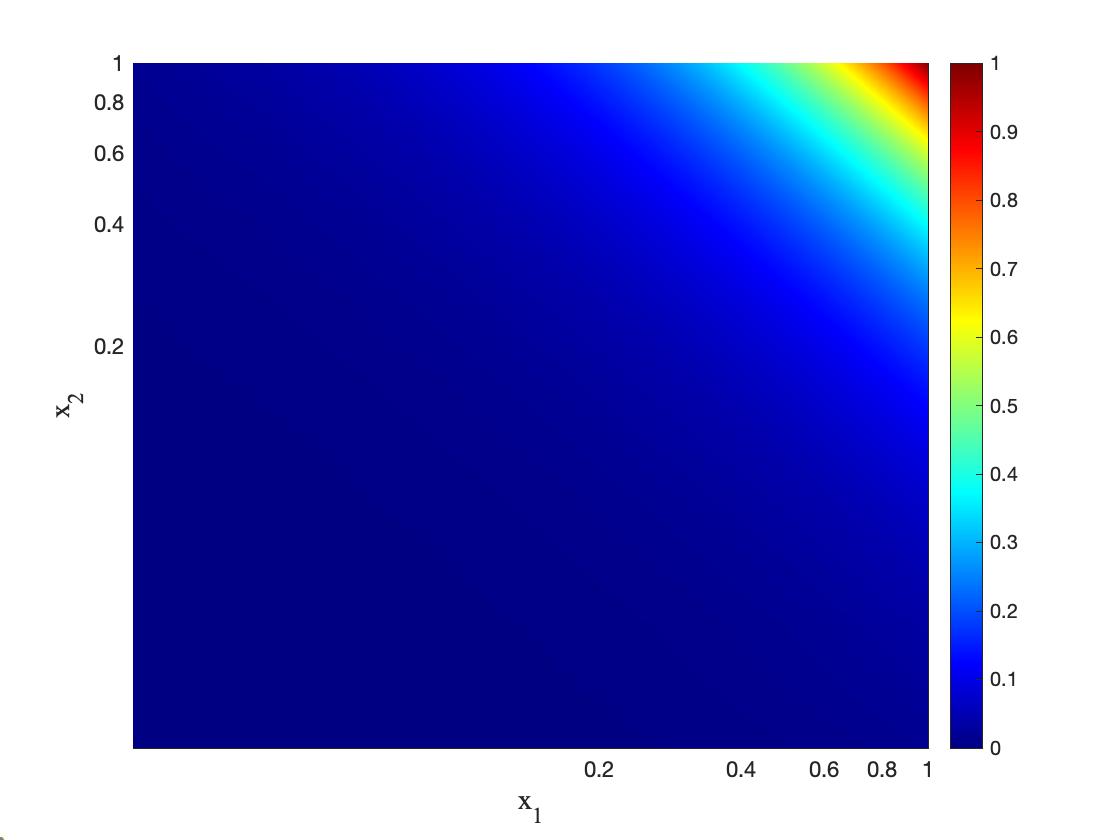}
\caption{True solution of $u(x,y) = x^{\frac{16}{15}}y^{\frac{16}{15}}$.}
\label{Fig:2d_true_xy_1516}
\end{figure}

\begin{figure}[H]
\centering
\subfigure[prediction while $\alpha=0.2$]{
\begin{minipage}[t]{0.25\linewidth}
\centering
\includegraphics[width=1.14\textwidth]{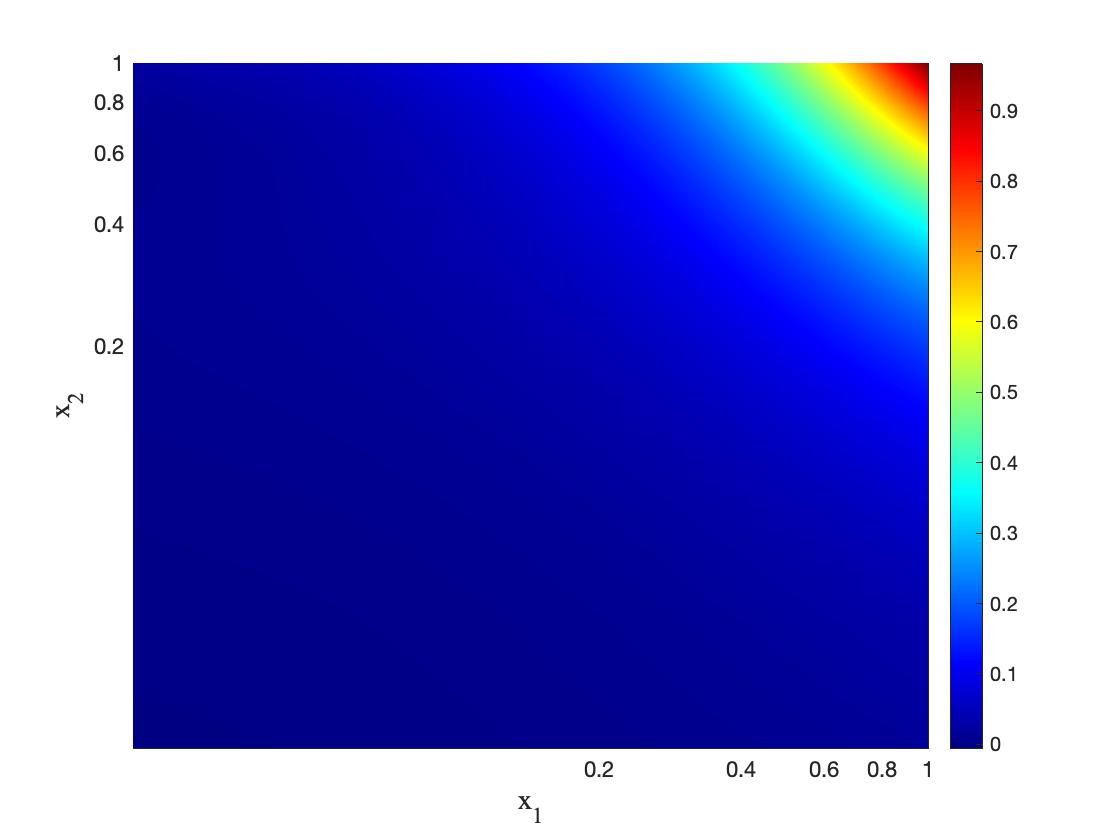}
\end{minipage}
}
\subfigure[prediction while $\alpha=0.4$]{
\begin{minipage}[t]{0.25\linewidth}
\centering
\includegraphics[width=1.14\textwidth]{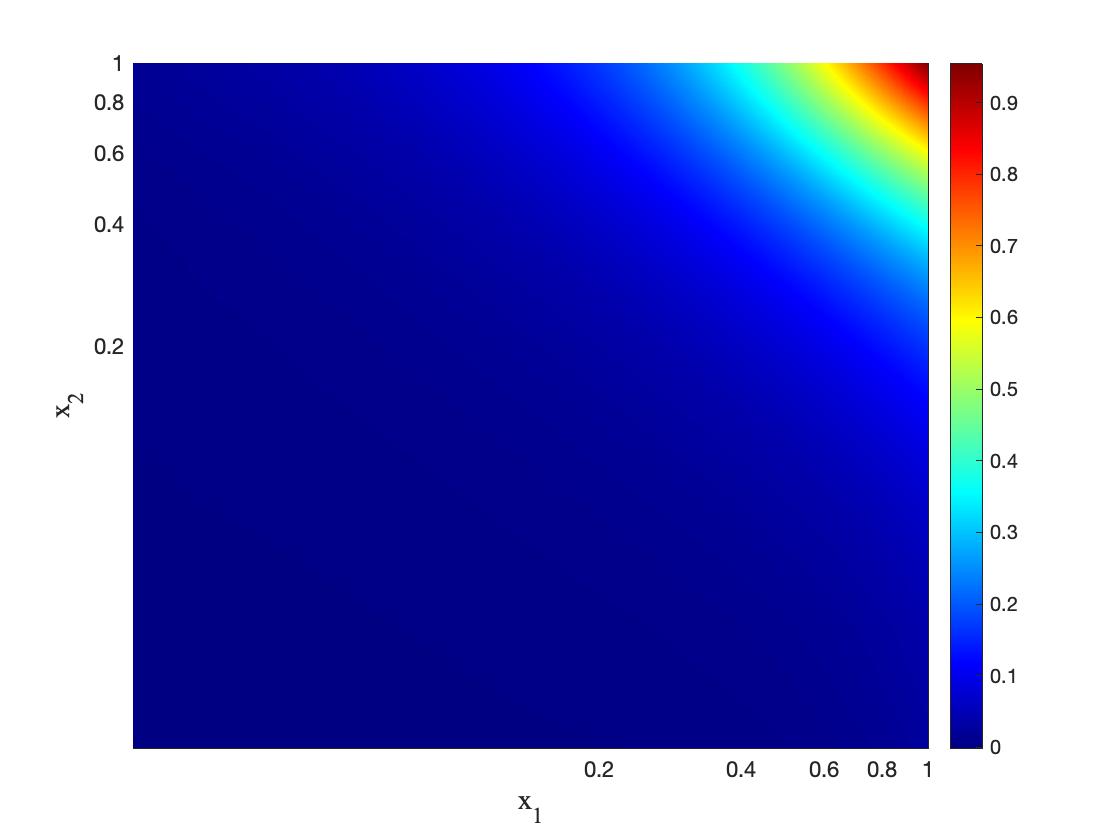}
\end{minipage}
}
\subfigure[prediction while $\alpha=0.6$]{
\begin{minipage}[t]{0.25\linewidth}
\centering
\includegraphics[width=1.14\textwidth]{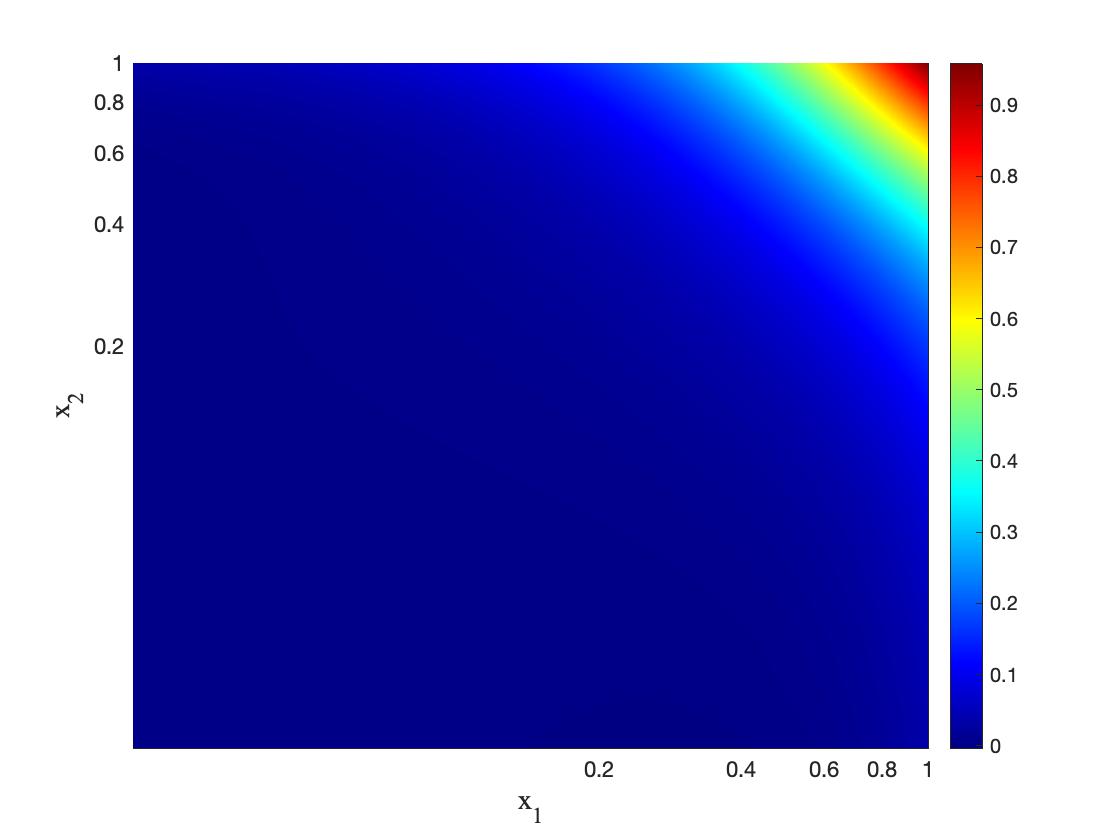}
\end{minipage}
}
                 
\subfigure[difference while $\alpha=0.2$]{
\begin{minipage}[t]{0.25\linewidth}
\centering
\includegraphics[width=1.14\textwidth]{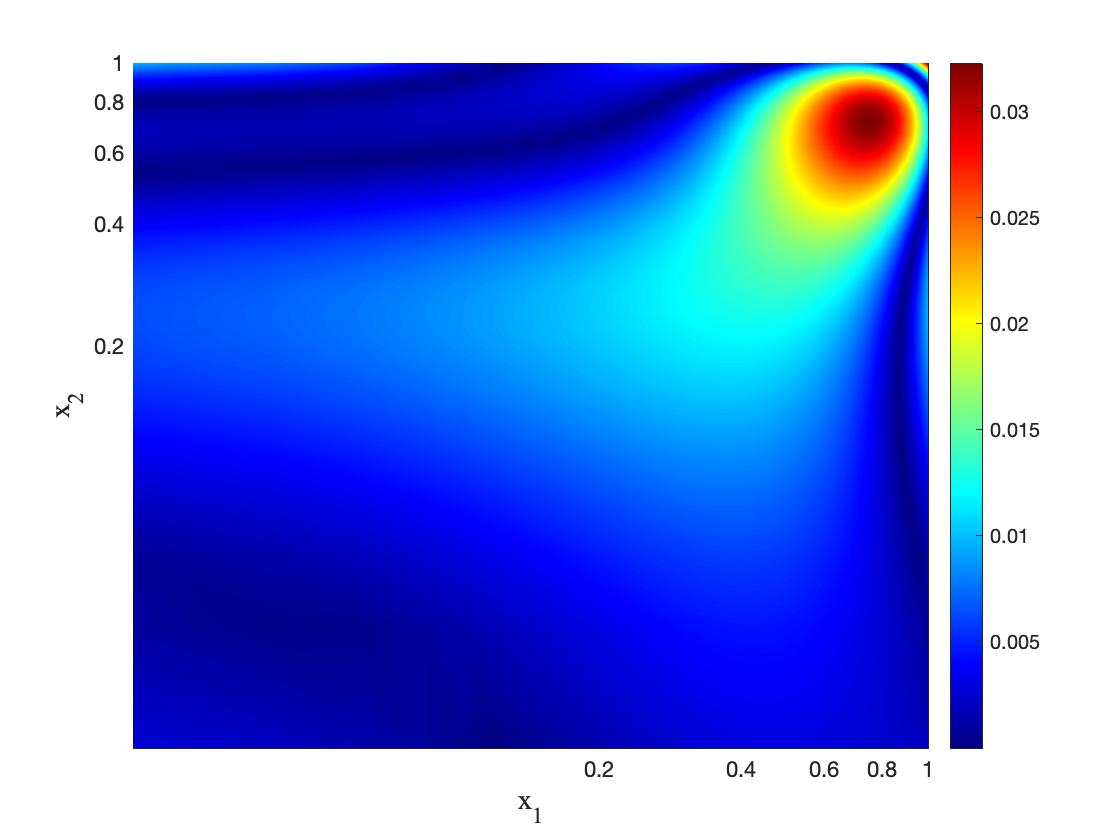}
\end{minipage}
}
\subfigure[difference while $\alpha=0.4$]{
\begin{minipage}[t]{0.25\linewidth}
\centering
\includegraphics[width=1.14\textwidth]{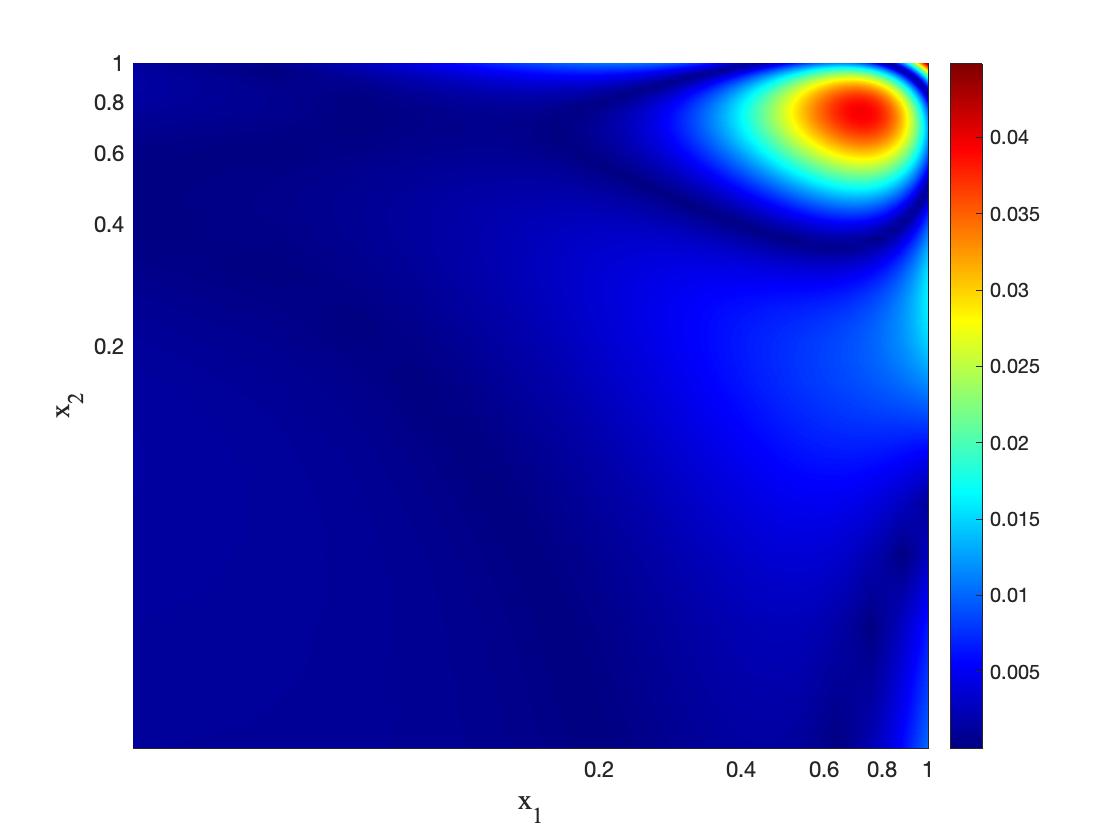}
\end{minipage}
}
\subfigure[difference while $\alpha=0.6$]{
\begin{minipage}[t]{0.25\linewidth}
\centering
\includegraphics[width=1.14\textwidth]{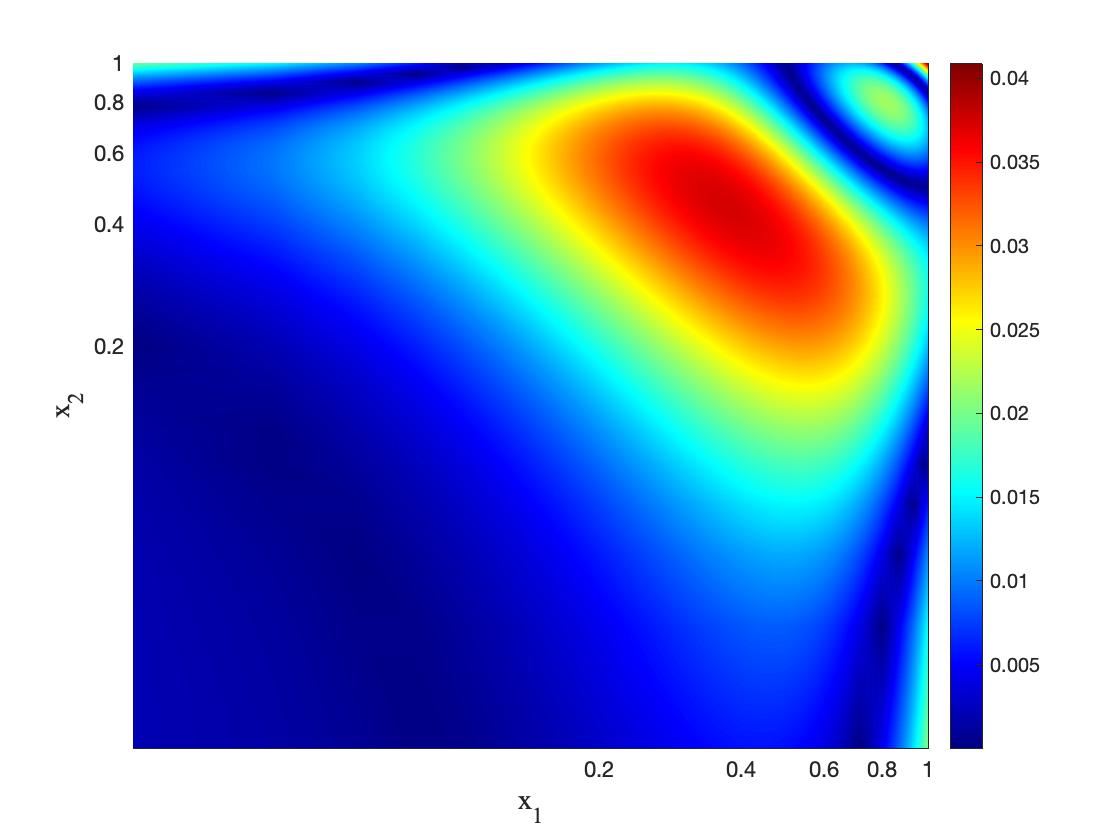}
\end{minipage}
}
\centering
\caption{Numerical solutions of the \ref{ex:lessSmooth} with less smoothness. In the first row, from left to right, is the predictive value of the WAN when alpha is equal to 0.2, 0.4 and 0.6. While the second row is the corresponding residual, i.e. $|u_{\text{prediction}} - u_{\text{true}}|$.
}
\label{Fig:2d_pred_xy}
\end{figure}

\subsubsection{The 3-dimensional fractional equation exhibits a solution with smoothness}\label{ex:3d}
In this subsection, we demonstrate the effectiveness of f-WANs in solving 3D problems. The higher dimensionality does not introduce any significant differences. Moreover, we will investigate the correlation between $\alpha$ and relative error.
Here we consider fractional equation with $d=3$ in (\ref{strong_form_int}) along with boundary condition (\ref{strong_form_bound}).
We consider the model problem in $\Omega=(0,1)^3$.
The exact solution is given by $u(x,y,z)=x^2 y^2z^2$.
First, in Figure \ref{Fig:3d_plot}, we give the volume slice planes of model problem with $\alpha=0.5$. We select the slice with $x=0.8$ and $z=0.7$ since the points with large values are primarily concentrated near the point $(1,1,1)$.

\begin{figure}[H]
\centering
\subfigure[$u_{\text{true}}$]{
\begin{minipage}[t]{0.4\linewidth}
\centering
\includegraphics[width=1\textwidth]{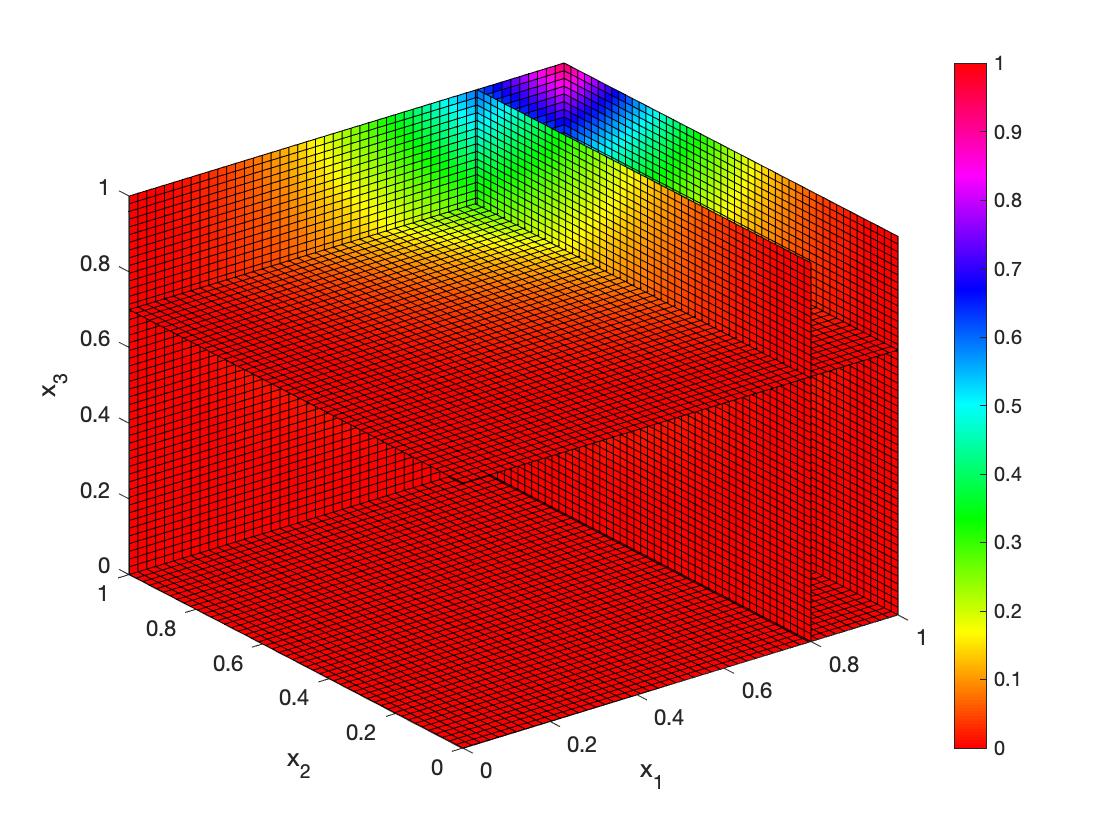}
\end{minipage}
}
\subfigure[$u_{\text{prediction}}$]{
\begin{minipage}[t]{0.4\linewidth}
\centering
\includegraphics[width=1\textwidth]{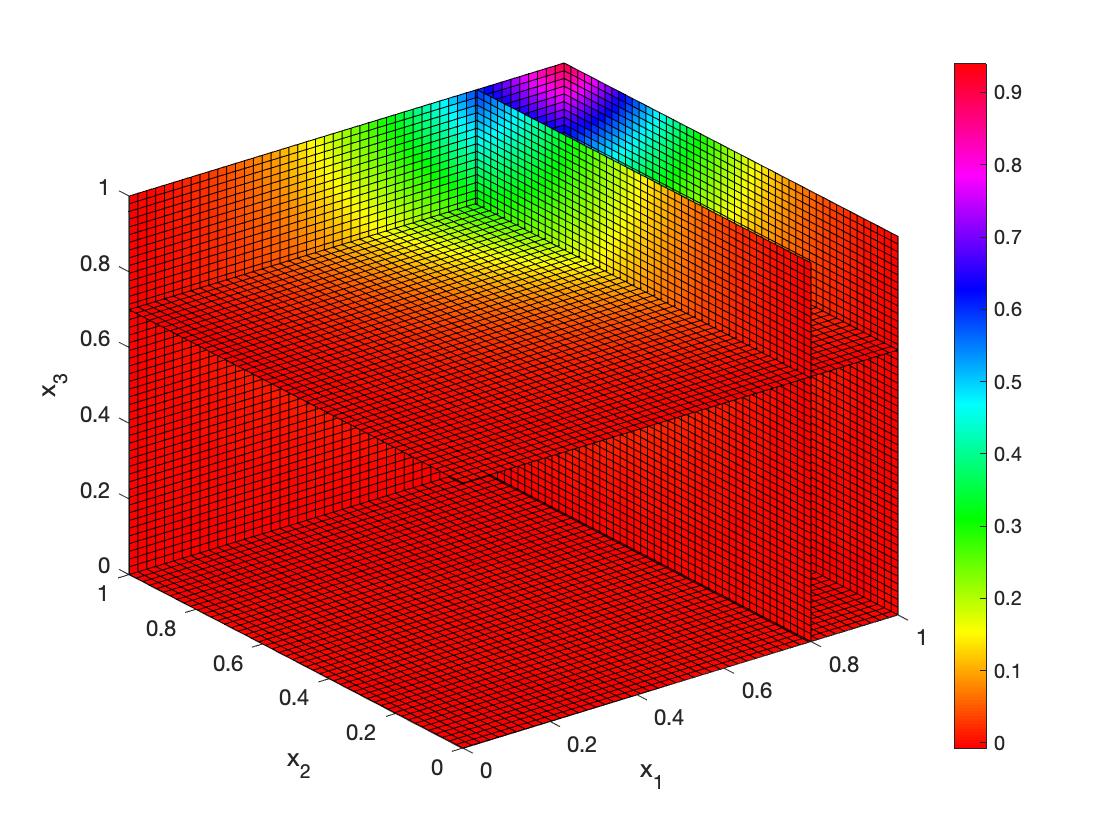}
\end{minipage}
}
\centering
\caption{Volume slice planes of model problem in 3-d with $\alpha=0.5$. The left image is the true value and the right one is the prediction. Both images have the same slice locations.}
\label{Fig:3d_plot}
\end{figure}

To better visualize the subtle difference between the true value and the prediction, we chose slices at $z=1$ and $z=x$, and plotted the two-dimensional images.
Figures \ref{Fig:3d_alpha_0.1} and \ref{Fig:3d_alpha_0.9} show the cases where $\alpha$ is equal to 0.1 and 0.9, respectively. In the first line of each figure, we show the slice of $z=1$, where the true value is $u(x,y)=x^2y^2$, and its projection on the $x-y$ plane. On the second line, we show the projection of the $x=z$ plane onto the $x-y$ plane, where the true value is $u(x,y)=x^4y^2$. Regardless of whether $\alpha$ is equal to 0.1 or 0.9, f-WANs performs well. Although the error near the point $(1,1)$ reached about $8\%$, the overall relative error is still around $3.5\%$, as given in equation \eqref{rela_err}.

\begin{figure}[H]
\centering
\subfigure[$u_{\text{true}}$ while $z=1$]{
\begin{minipage}[t]{0.25\linewidth}
\centering
\includegraphics[width=1.15\textwidth]{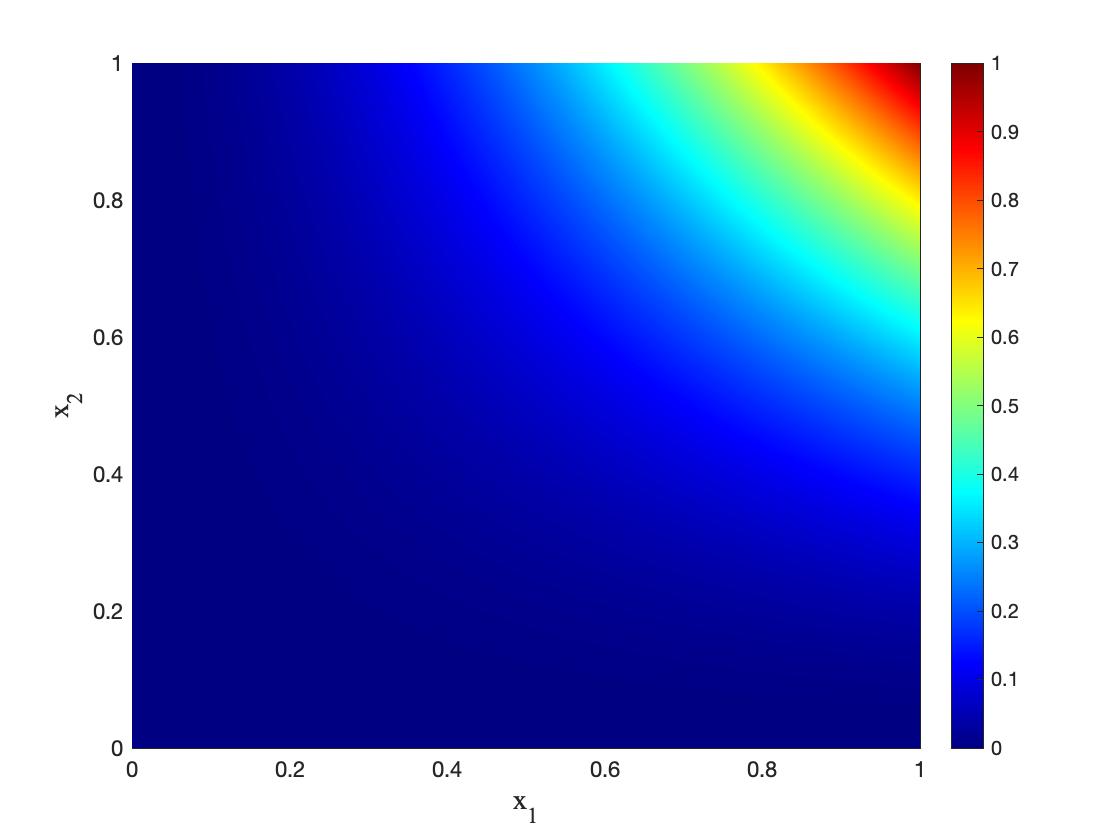}
\end{minipage}
}
\subfigure[$u_{\text{prediction}}$ while $z=1$]{
\begin{minipage}[t]{0.25\linewidth}
\centering
\includegraphics[width=1.15\textwidth]{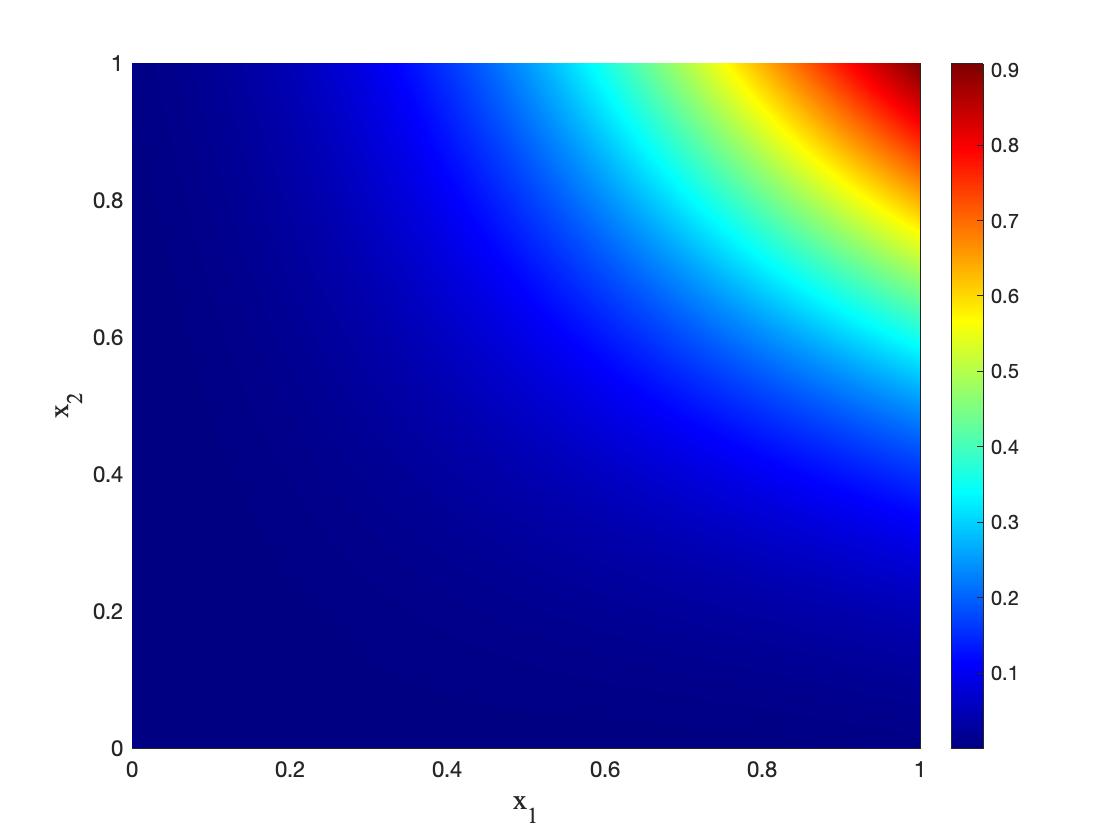}
\end{minipage}
}
\subfigure[$u_{\text{difference}}$ while $z=1$]{
\begin{minipage}[t]{0.25\linewidth}
\centering
\includegraphics[width=1.15\textwidth]{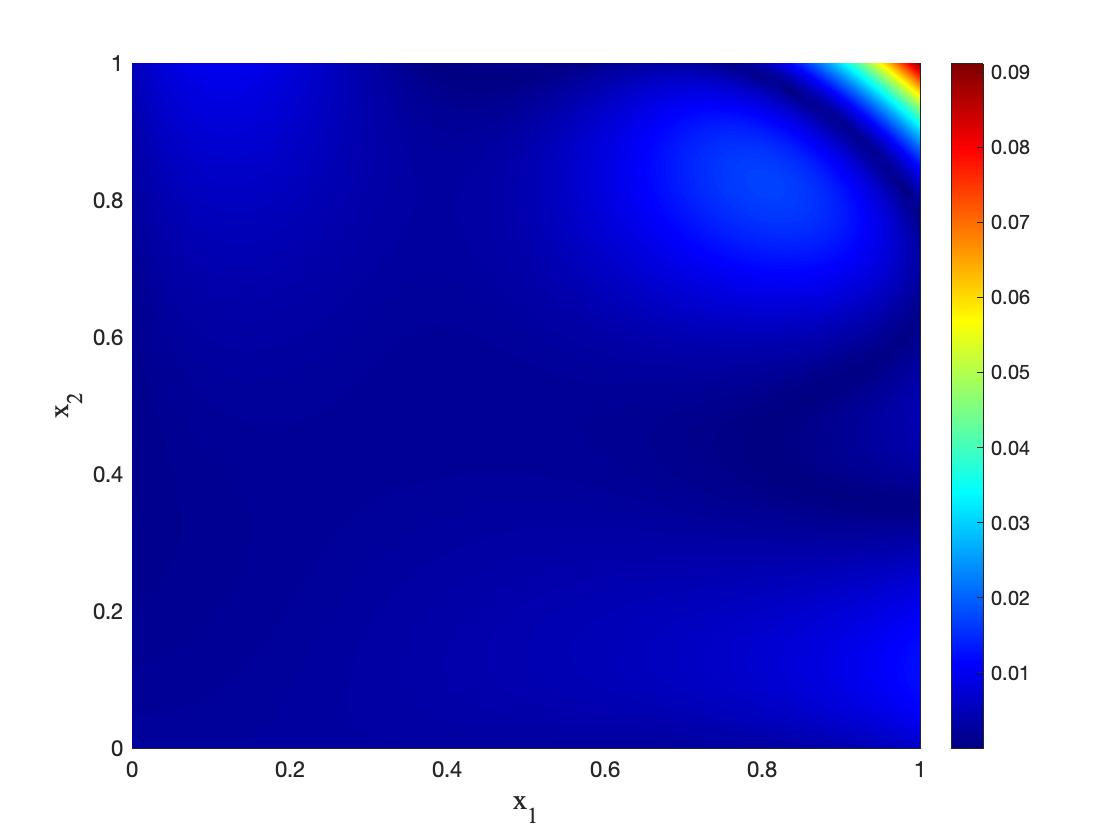}
\end{minipage}
}
                 
\subfigure[$u_{\text{true}}$ while $z=x$]{
\begin{minipage}[t]{0.25\linewidth}
\centering
\includegraphics[width=1.15\textwidth]{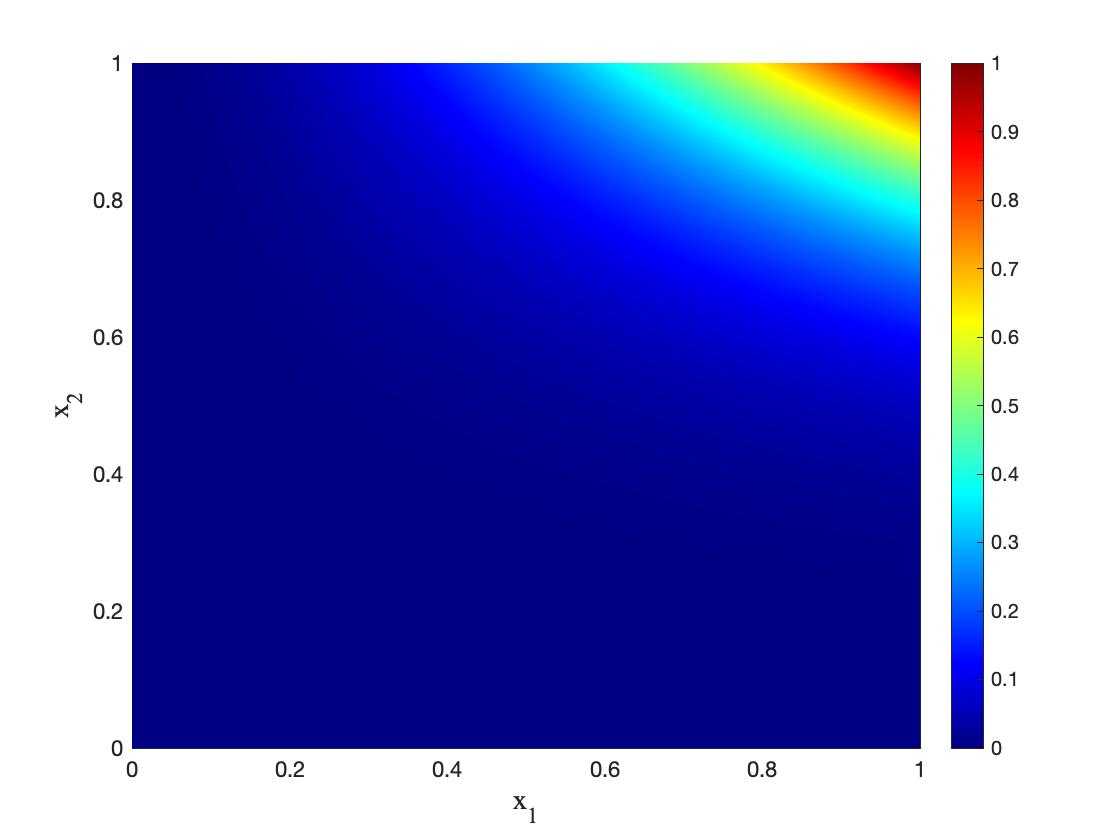}
\end{minipage}
}
\subfigure[$u_{\text{prediction}}$ while $z=x$]{
\begin{minipage}[t]{0.25\linewidth}
\centering
\includegraphics[width=1.15\textwidth]{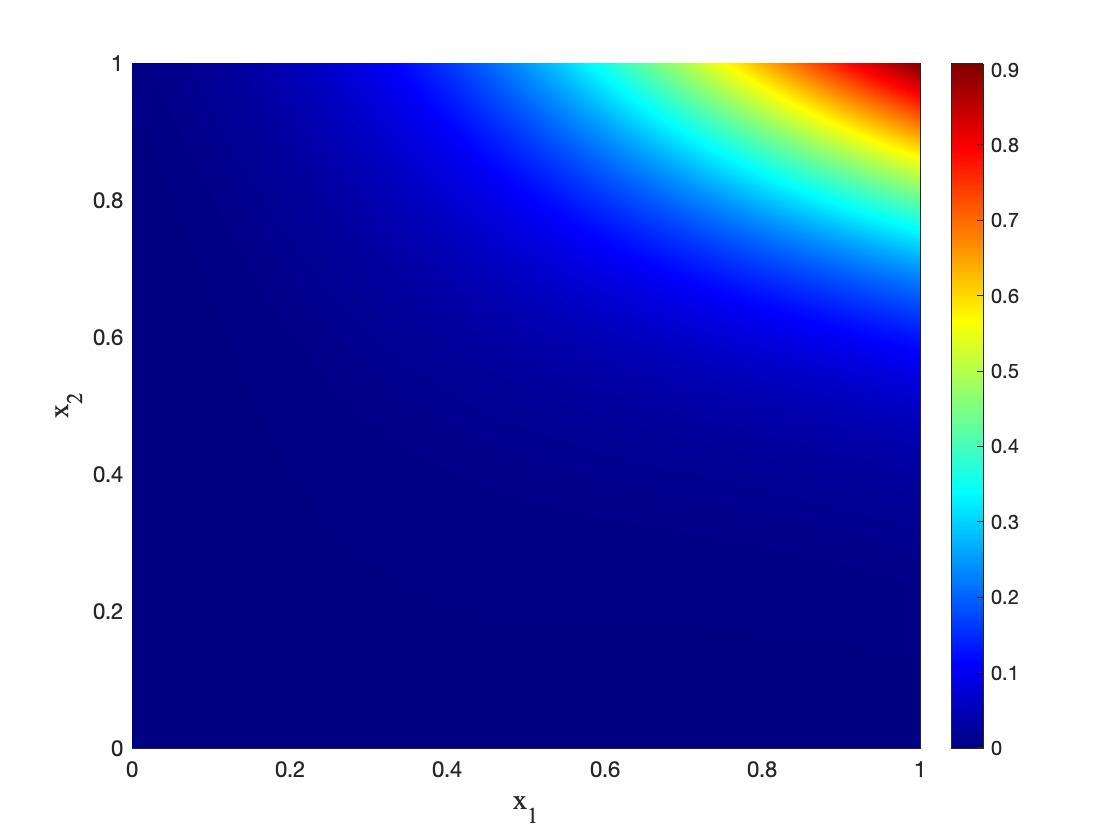}
\end{minipage}
}
\subfigure[$u_{\text{difference}}$ while $z=x$]{
\begin{minipage}[t]{0.25\linewidth}
\centering
\includegraphics[width=1.15\textwidth]{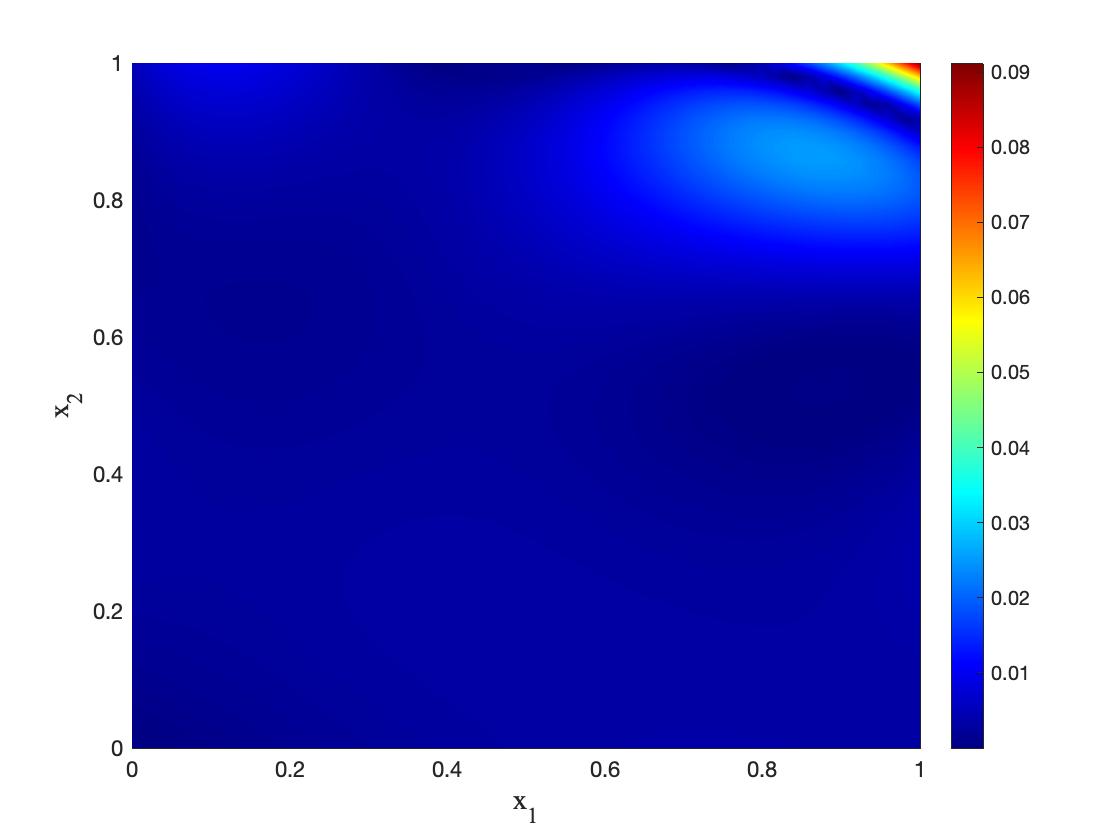}
\end{minipage}
}
\centering
\caption{Results of \ref{ex:3d} with $\alpha=0.1$. The first row displays the true value $u_{\text{true}}$, the predictive value $u_{\text{prediction}}$, and the difference $|u_{\text{true}}-u_{\text{prediction}}|$ while the slice is $z=1$. The second row presents the corresponding images with slice $z=x$.}
\label{Fig:3d_alpha_0.1}
\end{figure}

\begin{figure}[H]
\centering
\subfigure[$u_{\text{true}}$ while $z=1$]{
\begin{minipage}[t]{0.25\linewidth}
\centering
\includegraphics[width=1.15\textwidth]{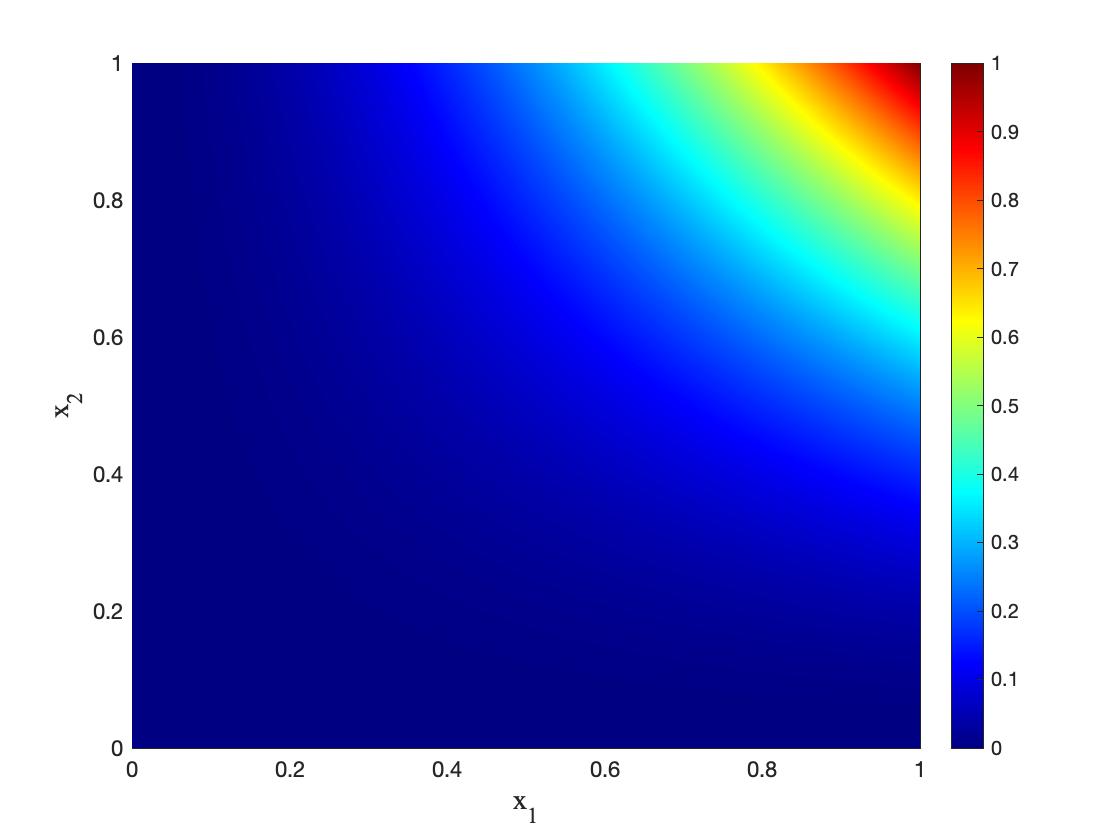}
\end{minipage}
}
\subfigure[$u_{\text{prediction}}$ while $z=1$]{
\begin{minipage}[t]{0.25\linewidth}
\centering
\includegraphics[width=1.15\textwidth]{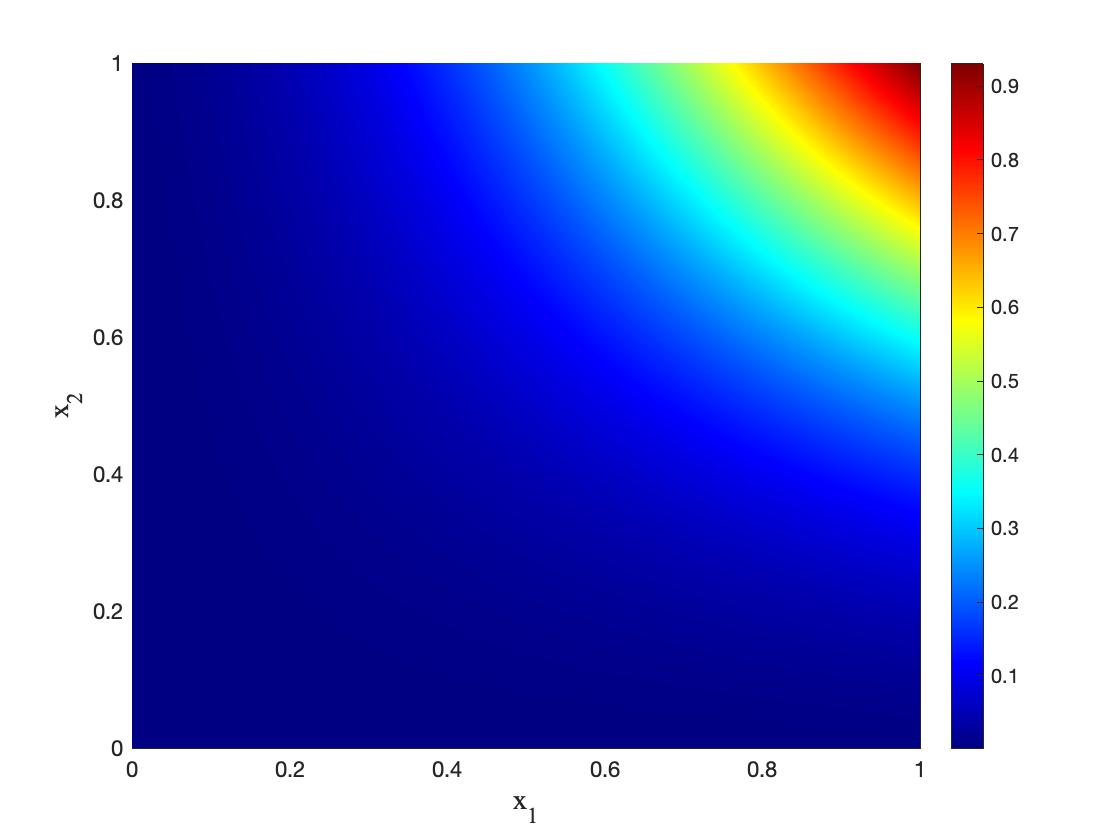}
\end{minipage}
}
\subfigure[$u_{\text{difference}}$ while $z=1$]{
\begin{minipage}[t]{0.25\linewidth}
\centering
\includegraphics[width=1.15\textwidth]{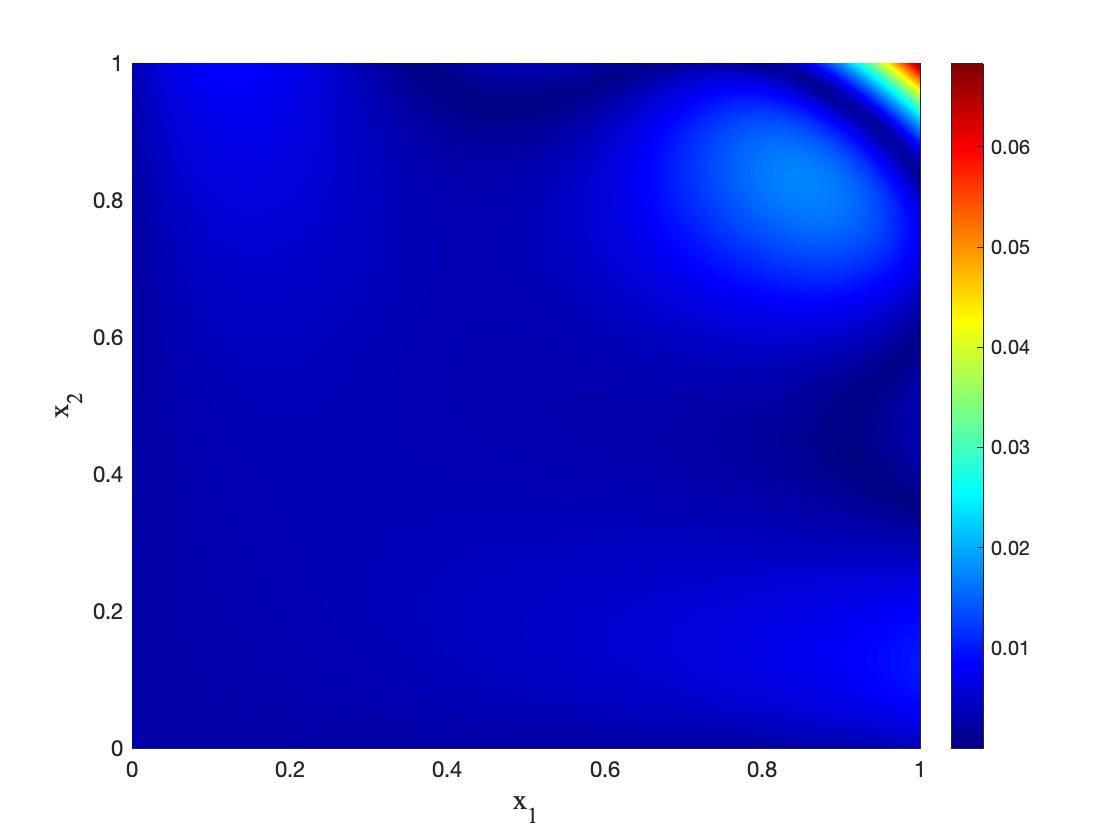}
\end{minipage}
}
                 
\subfigure[$u_{\text{true}}$ while $z=x$]{
\begin{minipage}[t]{0.25\linewidth}
\centering
\includegraphics[width=1.15\textwidth]{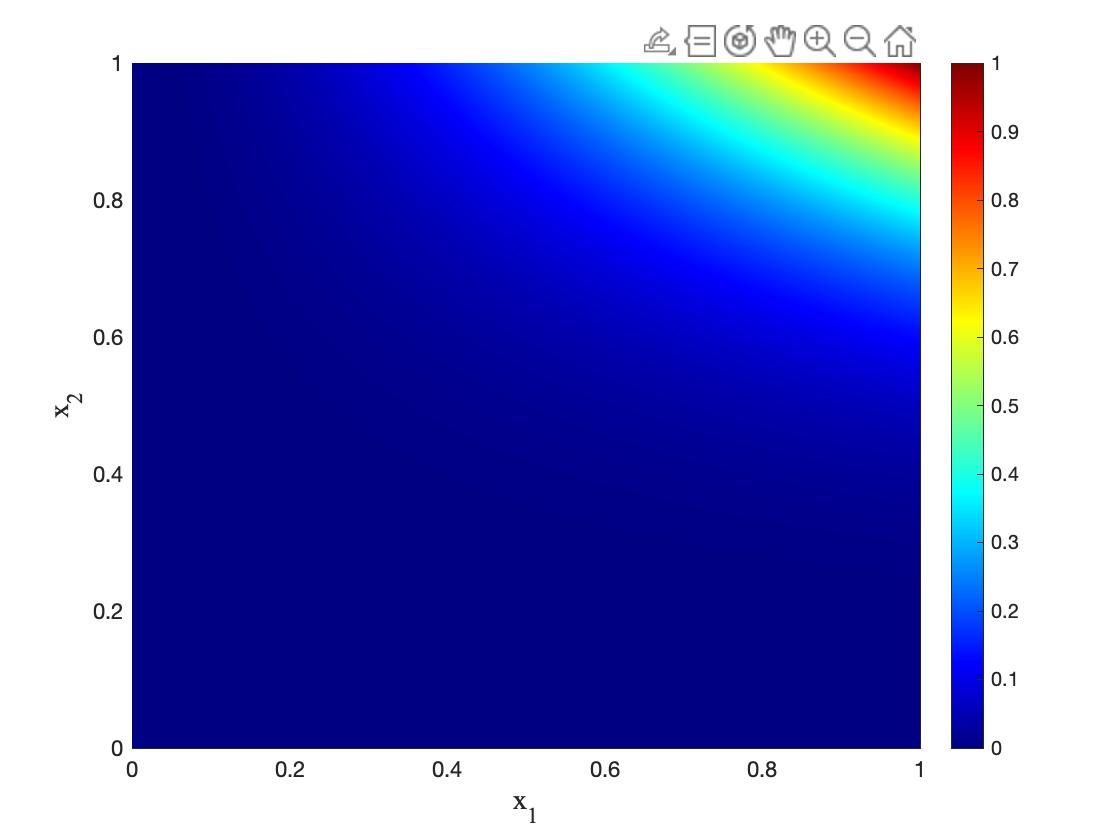}
\end{minipage}
}
\subfigure[$u_{\text{prediction}}$ while $z=x$]{
\begin{minipage}[t]{0.25\linewidth}
\centering
\includegraphics[width=1.15\textwidth]{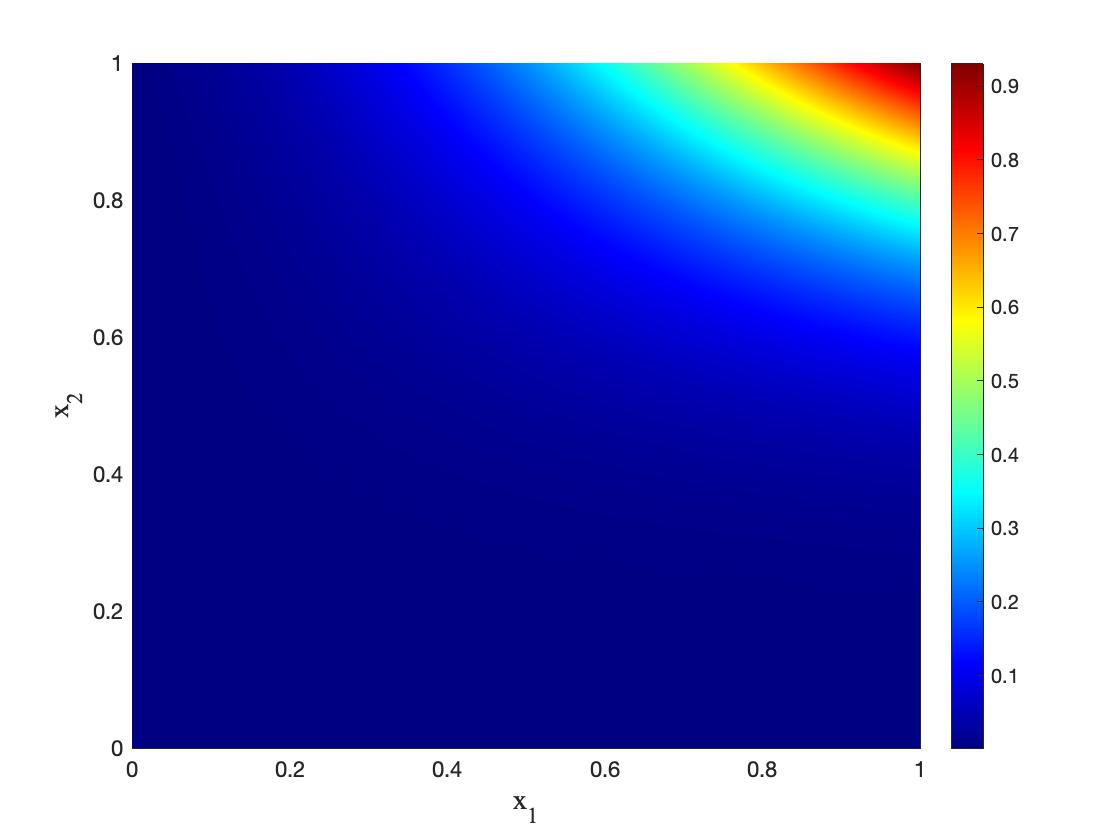}
\end{minipage}
}
\subfigure[$u_{\text{difference}}$ while $z=x$]{
\begin{minipage}[t]{0.25\linewidth}
\centering
\includegraphics[width=1.15\textwidth]{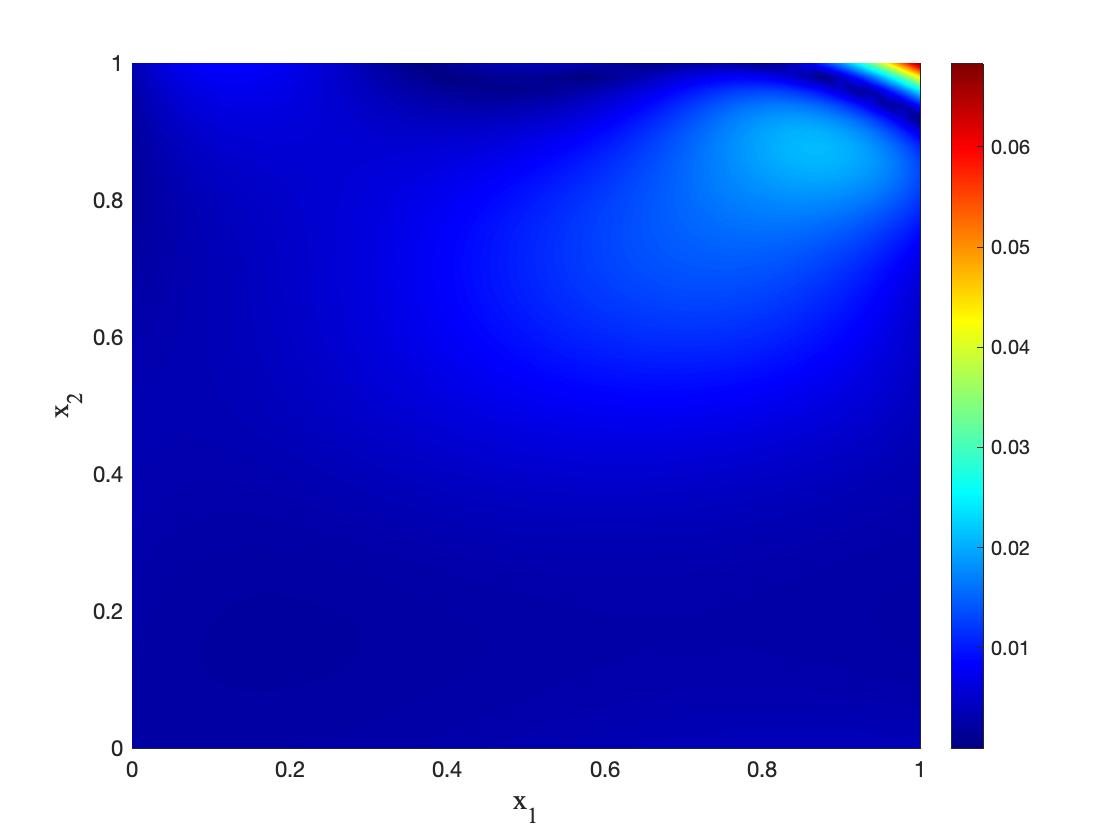}
\end{minipage}
}
\centering
\caption{Results of \ref{ex:3d} with $\alpha=0.9$. The first row displays the true value $u_{\text{true}}$, the predictive value $u_{\text{prediction}}$, and the difference $|u_{\text{true}}-u_{\text{prediction}}|$ while the slice is $z=1$. The second row presents the corresponding images with slice $z=x$.}
\label{Fig:3d_alpha_0.9}
\end{figure}

Finally, we present the relationship between $\alpha$ and relative error in Figure \ref{rela_err_pic}. This figure is generated using the same parameter settings, and we repeat the experiment ten times, taking the average of the results. In equation \eqref{strong_form_int}, there is a term $|x-w|^{\alpha - 1}$, where the exponent $\alpha-1 \in (-1,0)$ since $\alpha \in (0,1)$. As $\alpha$ approaches zero, the singularity caused by $|x-w|^{\alpha-1}$ becomes more pronounced.
The relative error, measured in the $\ell^2$-norm, is given by:
\begin{equation}\label{rela_err}
\frac{\| u_{prediction} - u_{true} \|_{\ell^2}}{\| u_{true} \|_{\ell^2}}.
\end{equation}

\begin{figure}[H]
\centering
\includegraphics[width=0.35\textwidth]{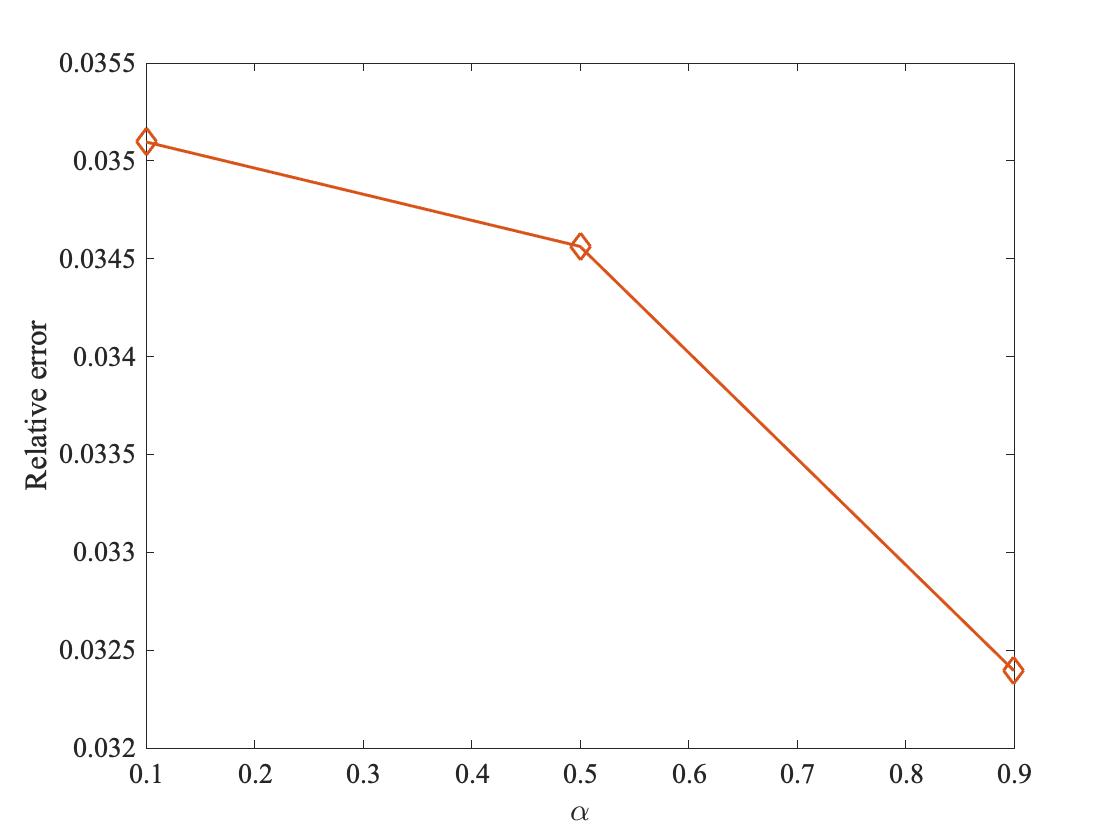}
\caption{Relative error for different values of $\alpha$.}
\label{rela_err_pic}
\end{figure}

\subsubsection{The 3-dimensional fractional equation exhibits a solution with less smoothness}\label{ex:3d:lessSmooth}

In this example, we choose the true solution by $u(x,y,z)=x^{\frac{16}{15}}+y^{\frac{16}{15}}+z^{\frac{16}{15}}$ in three dimensional case.
To enhance the expressive power of the network, we increase the number of neurons in each hidden layer of the neural network $u_\theta$ to 40 while keeping the rest of the structure the same.
We choose the slice at $y=0.7$ and $z=0.8$. We set $\alpha=0.4$. The volume slice planes are shown in Figure \ref{Fig:3d_plot_x1_x}, while the projection of the true value onto the plane $z=1$ and $z=x$ with logarithmic scale for coordinates is shown in Figure \ref{3d_x1_x}.

\begin{figure}[H]
\centering
\subfigure[$u_{\text{true}}$]{
\begin{minipage}[t]{0.4\linewidth}
\centering
\includegraphics[width=1\textwidth]{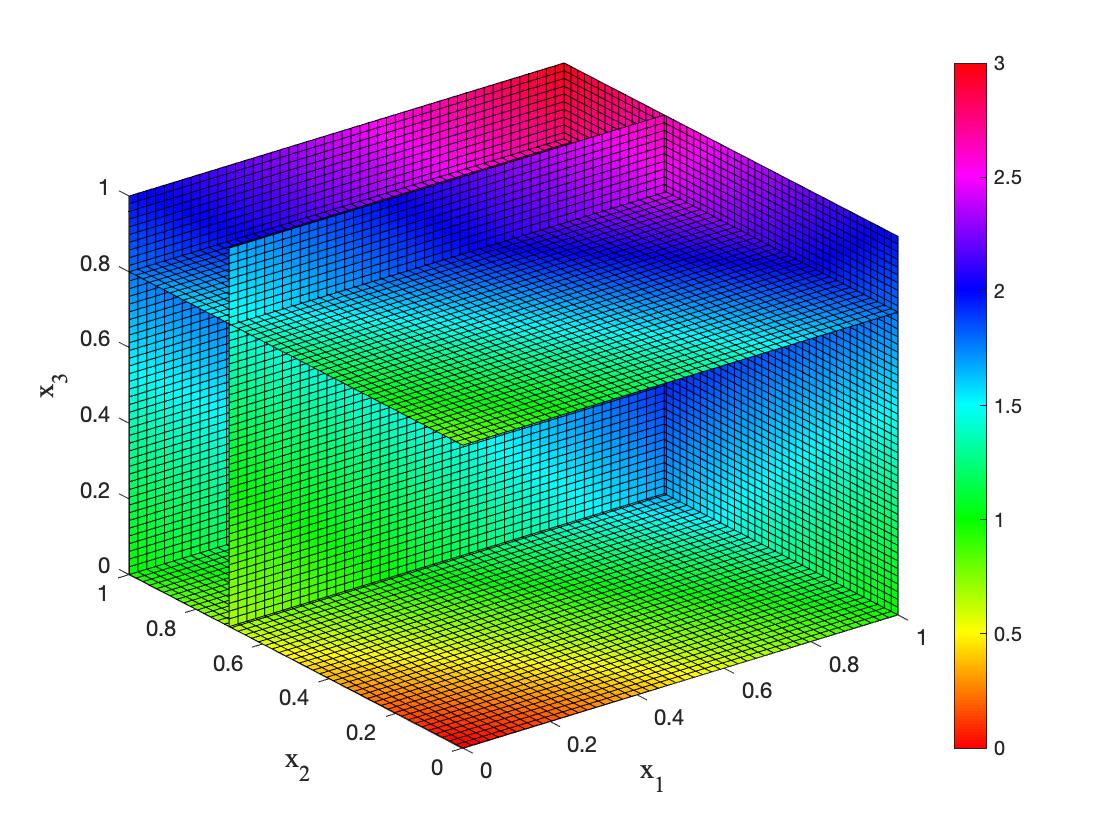}
\end{minipage}
}
\subfigure[$u_{\text{prediction}}$]{
\begin{minipage}[t]{0.4\linewidth}
\centering
\includegraphics[width=1\textwidth]{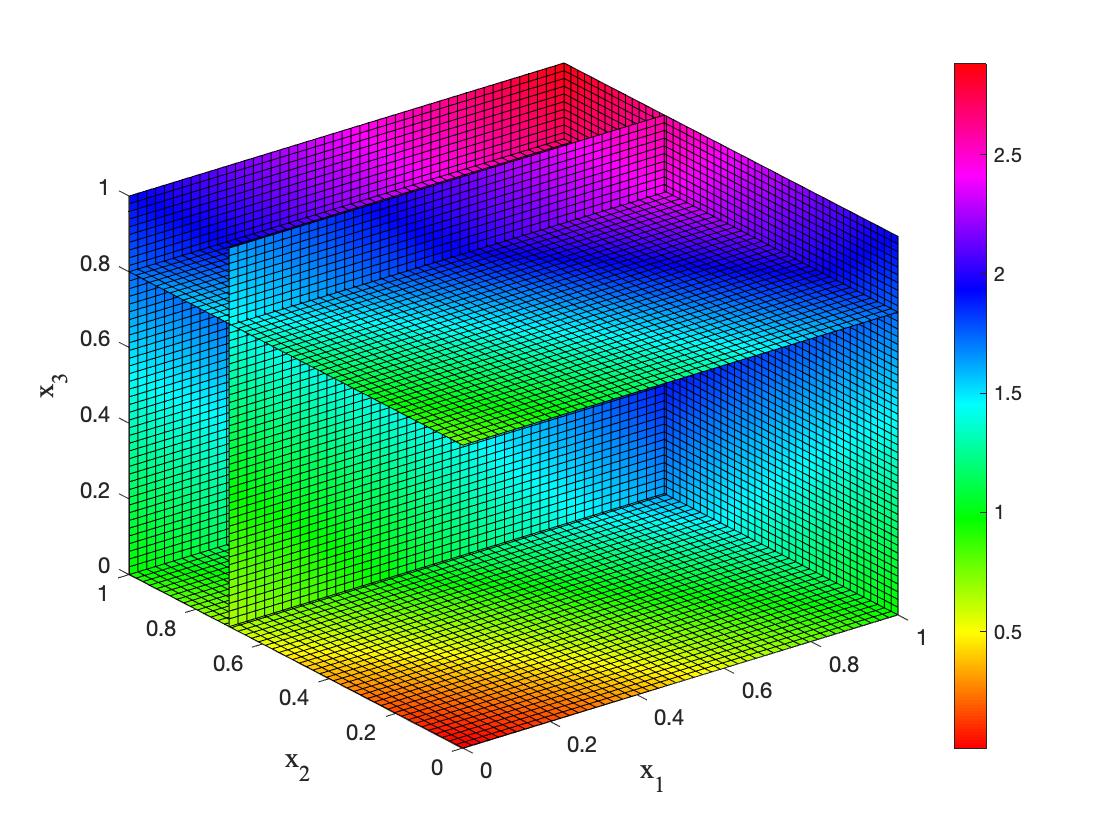}
\end{minipage}
}
\centering
\caption{Volume slice planes \ref{ex:3d:lessSmooth}  with $\alpha=0.4$ . The left image is the true value and the right one is the prediction. Both images have the same slice locations.}
\label{Fig:3d_plot_x1_x}
\end{figure}

\begin{figure}[H]
\centering
\subfigure[$u_{\text{true}}$ while $z=1$]{
\begin{minipage}[t]{0.25\linewidth}
\centering
\includegraphics[width=1.15\textwidth]{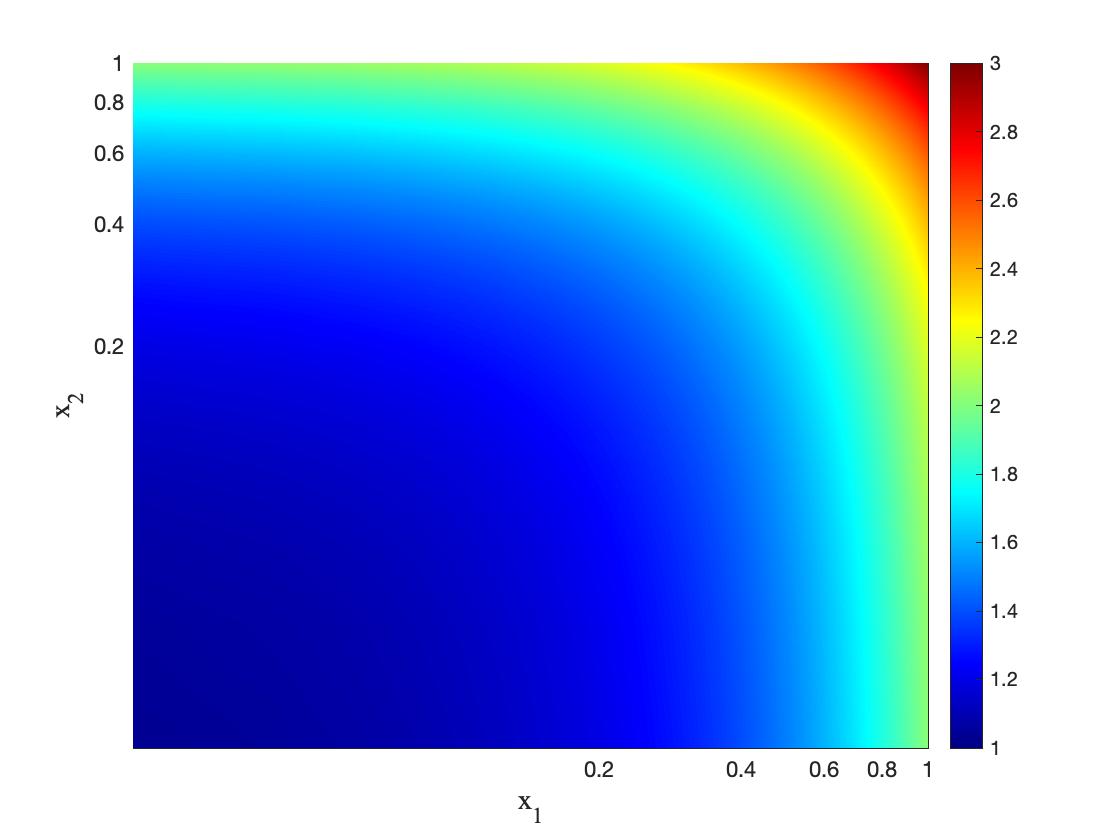}
\end{minipage}
}
\subfigure[$u_{\text{prediction}}$ while $z=1$]{
\begin{minipage}[t]{0.25\linewidth}
\centering
\includegraphics[width=1.15\textwidth]{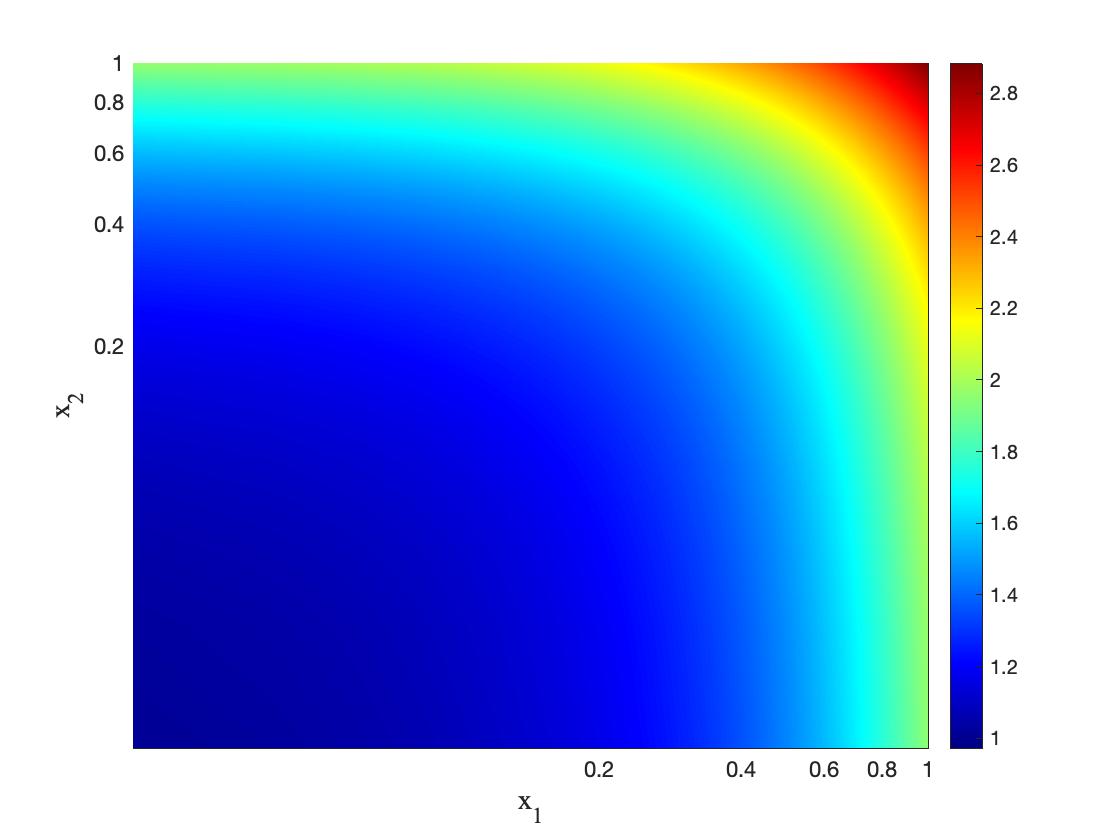}
\end{minipage}
}
\subfigure[$u_{\text{difference}}$ while $z=1$]{
\begin{minipage}[t]{0.25\linewidth}
\centering
\includegraphics[width=1.15\textwidth]{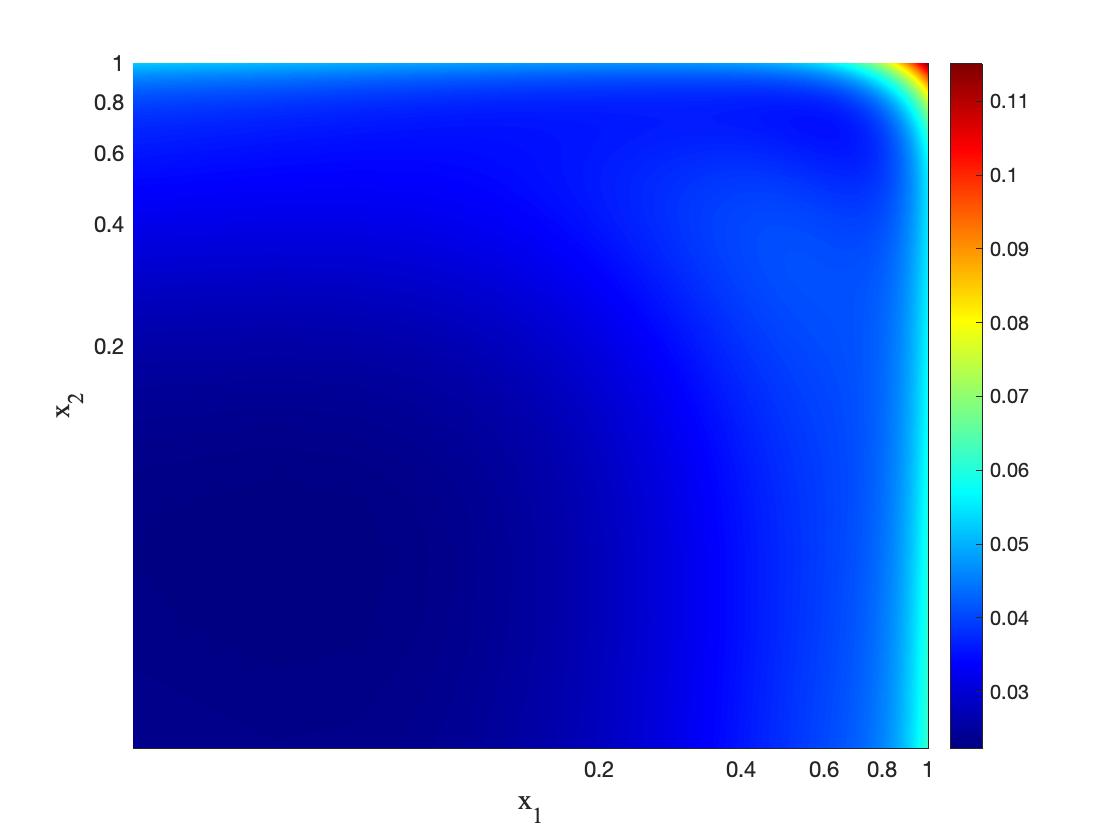}
\end{minipage}
}

\subfigure[$u_{\text{true}}$ while $z=x$]{
\begin{minipage}[t]{0.25\linewidth}
\centering
\includegraphics[width=1.15\textwidth]{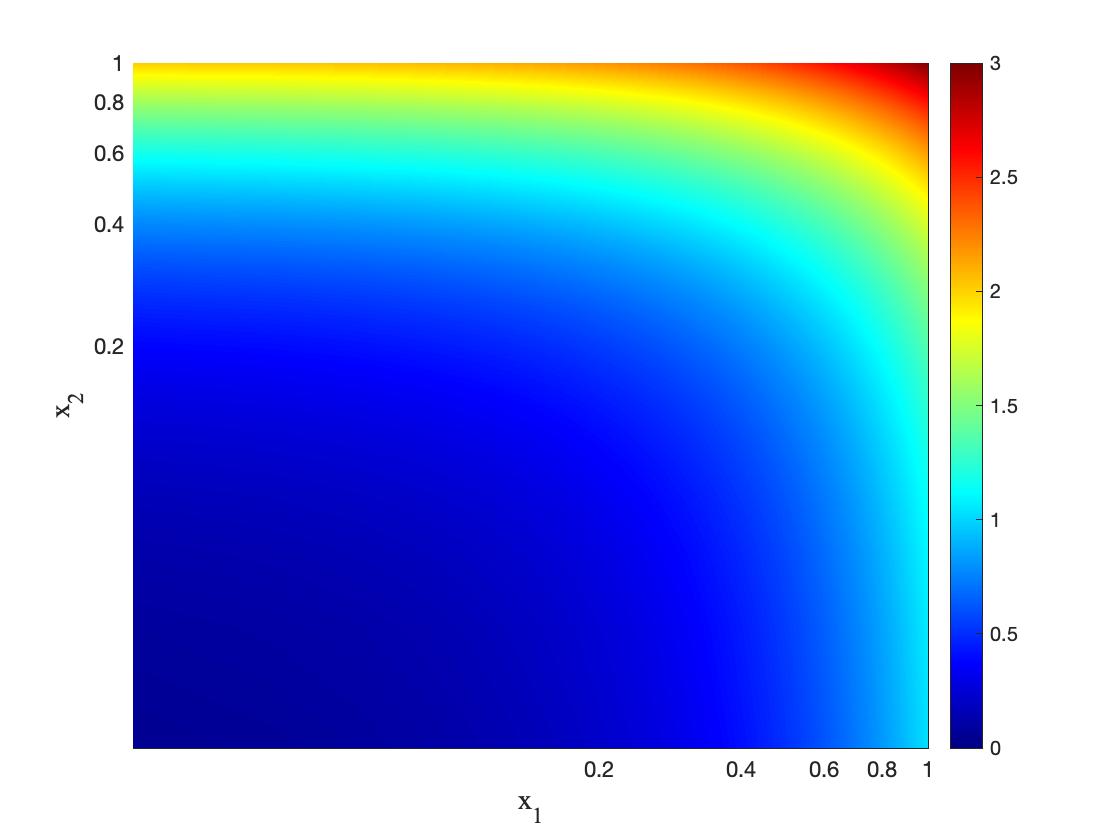}
\end{minipage}
}
\subfigure[$u_{\text{prediction}}$ while $z=x$]{
\begin{minipage}[t]{0.25\linewidth}
\centering
\includegraphics[width=1.15\textwidth]{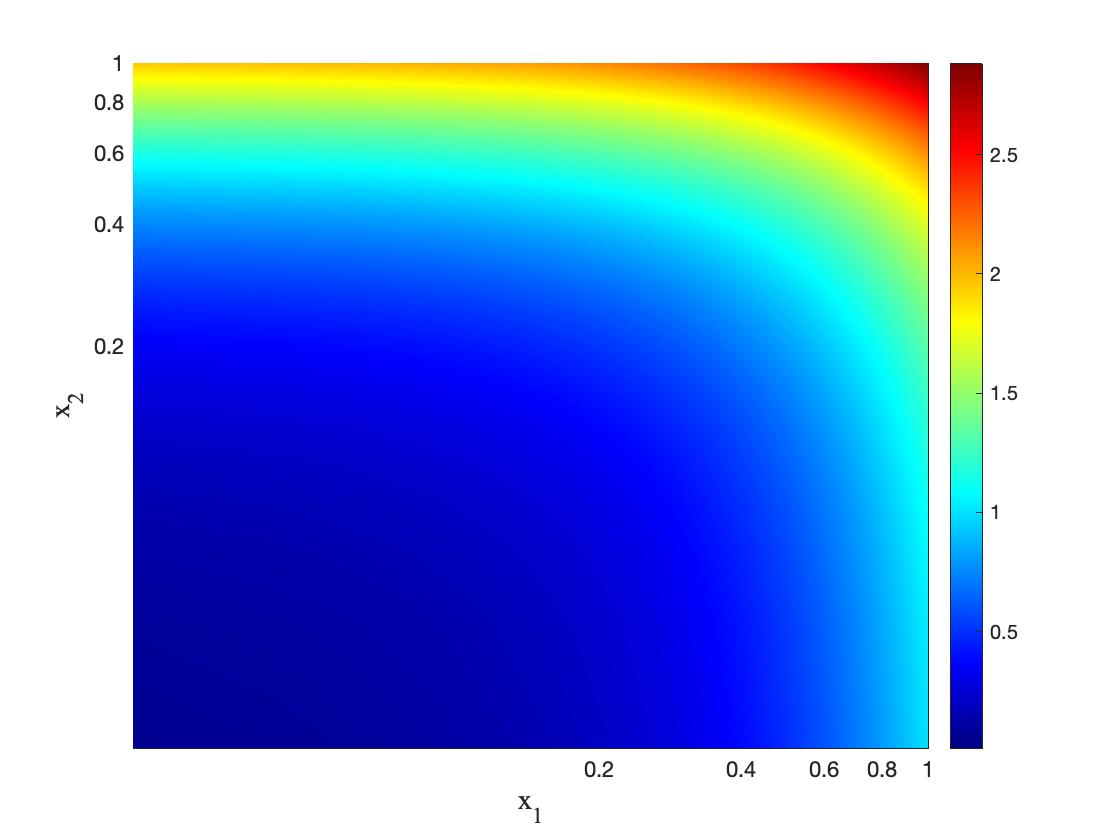}
\end{minipage}
}
\subfigure[$u_{\text{difference}}$ while $z=x$]{
\begin{minipage}[t]{0.25\linewidth}
\centering
\includegraphics[width=1.15\textwidth]{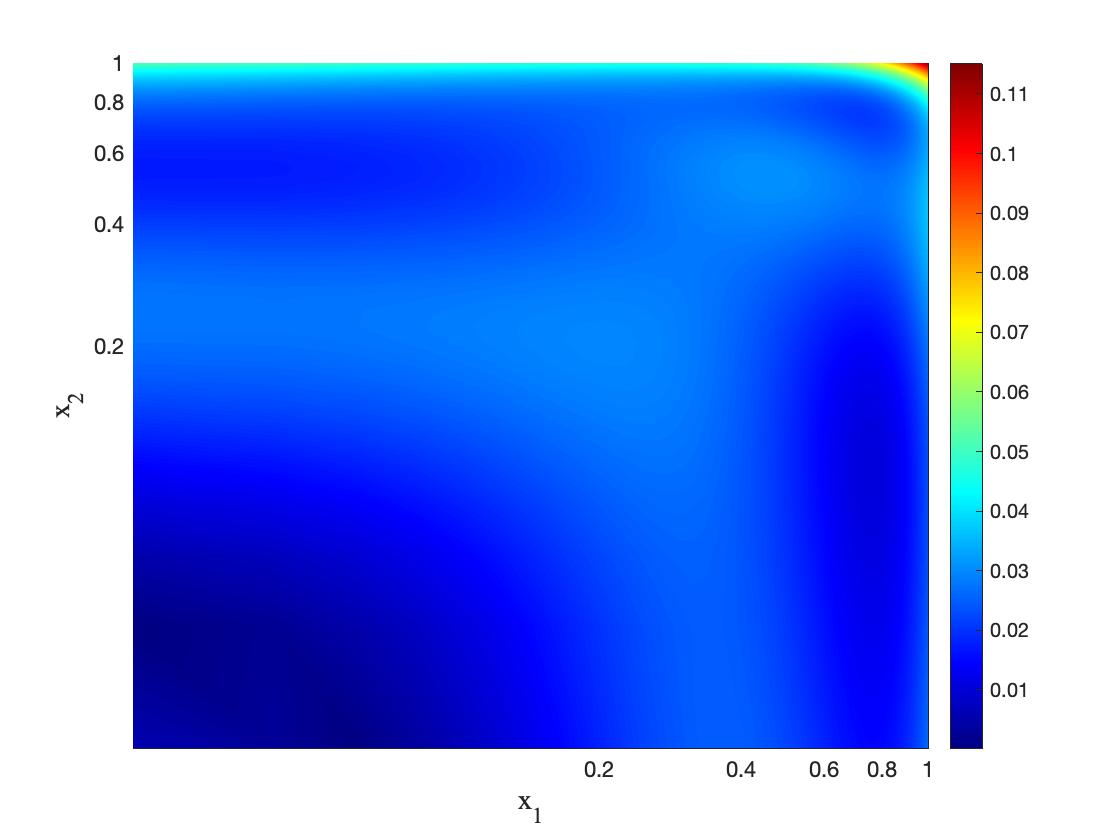}
\end{minipage}
}
\centering
\caption{Results of  \ref{ex:3d:lessSmooth} with $\alpha=0.4$. The first row displays the true value $u_{\text{true}}$, the predictive value $u_{\text{prediction}}$, and the difference $|u_{\text{true}}-u_{\text{prediction}}|$ while the slice is $z=1$. The second row presents the corresponding images with slice $z=x$.}
\label{3d_x1_x}
\end{figure}

\section{Concluding remarks}\label{Sect:Conc}
In this paper, we propose a novel structure of neural network called f-WANs, based on the weak form of the FADE, which shows its efficiency to handle both smooth and less smooth solutions in high dimensions. Our approach combines Monte Carlo sampling method with the neural network to approximate the solution of the fractional differential equation, which can be extend to general fractional differential equations that admits a variational form. Our experiments focus on 2D and 3D problems defined on rectangle domains, but it's possible to  extend our proposed architectures to the general convex bounded Lipschitz domains by carefully reparametrizing the domain, although this may involve technical considerations and efforts.

% ------------------------------------------------------------------------
\end{document}